\documentclass[11pt]{amsart}
\usepackage{amscd}
\usepackage[arrow,matrix]{xy}
\usepackage{graphicx}
\usepackage{amsmath}
\usepackage{xcolor}
\usepackage{soul}
\usepackage{comment}
\usepackage{amsmath, latexsym, amssymb}
\input xypic
\numberwithin{equation}{section}
\theoremstyle{plain}
\newtheorem{lemma}{Lemma}[section]
\newtheorem{proposition}[lemma]{Proposition}
\newtheorem{theorem}[lemma]{Theorem}
\newtheorem{corollary}[lemma]{Corollary}

\theoremstyle{definition}
\newtheorem{definition}[lemma]{Definition}
\newtheorem{remark}[lemma]{Remark}
\newtheorem{example}[lemma]{Example}

\definecolor{purple}{RGB}{150,0,100}

\newcommand{\vol}{\mbox{\rm vol}}

\renewcommand{\Im}{{ \rm Im \,}}
\renewcommand{\Re}{{\rm Re \,}}

\newcommand{\T}{\mathcal T}

\begin{document}
\newcommand{\R}{{\mathbb R}}
\newcommand{\C}{{\mathbb C}}
\newcommand{\F}{{\mathbb F}}
\renewcommand{\O}{{\mathbb O}}
\newcommand{\Z}{{\mathbb Z}} 
\newcommand{\N}{{\mathbb N}}
\newcommand{\Q}{{\mathbb Q}}
\renewcommand{\H}{{\mathbb H}}
\newcommand{\X}{{\mathfrak X}}

\newcommand{\Aa}{{\mathcal A}}
\newcommand{\Bb}{{\mathcal B}}
\newcommand{\Cc}{{\mathcal C}}    
\newcommand{\Dd}{{\mathcal D}}
\newcommand{\Ee}{{\mathcal E}}
\newcommand{\Ff}{{\mathcal F}}
\newcommand{\Gg}{{\mathcal G}}    
\newcommand{\Hh}{{\mathcal H}}
\newcommand{\Kk}{{\mathcal K}}
\newcommand{\Ii}{{\mathcal I}}
\newcommand{\Jj}{{\mathcal J}}
\newcommand{\Ll}{{\mathcal L}}    
\newcommand{\Mm}{{\mathcal M}}    
\newcommand{\Nn}{{\mathcal N}}
\newcommand{\Oo}{{\mathcal O}}
\newcommand{\Pp}{{\mathcal P}}
\newcommand{\Qq}{{\mathcal Q}}
\newcommand{\Rr}{{\mathcal R}}
\newcommand{\Ss}{{\mathcal S}}
\newcommand{\Tt}{{\mathcal T}}
\newcommand{\Uu}{{\mathcal U}}
\newcommand{\Vv}{{\mathcal V}}
\newcommand{\Ww}{{\mathcal W}}
\newcommand{\Xx}{{\mathcal X}}
\newcommand{\Yy}{{\mathcal Y}}
\newcommand{\Zz}{{\mathcal Z}}

\renewcommand{\a}{{\mathfrak a}}
\renewcommand{\b}{{\mathfrak b}}
\newcommand{\e}{{\mathfrak e}}
\renewcommand{\k}{{\mathfrak k}}
\newcommand{\m}{{\mathfrak m}}
\newcommand{\pg}{{\mathfrak p}}
\newcommand{\g}{{\mathfrak g}}
\newcommand{\gl}{{\mathfrak gl}}
\newcommand{\h}{{\mathfrak h}}
\renewcommand{\l}{{\mathfrak l}}
\newcommand{\sm}{{\mathfrak m}}
\newcommand{\n}{{\mathfrak n}}
\newcommand{\s}{{\mathfrak s}}
\renewcommand{\o}{{\mathfrak o}}
\renewcommand{\u}{{\mathfrak u}}
\newcommand{\su}{{\mathfrak su}}

\newcommand{\ssl}{{\mathfrak sl}}
\newcommand{\ssp}{{\mathfrak sp}}
\renewcommand{\t}{{\mathfrak t }}

\newcommand{\zt}{{\tilde z}}
\newcommand{\xt}{{\tilde x}}
\newcommand{\Ht}{\widetilde{H}}
\newcommand{\ut}{{\tilde u}}
\newcommand{\Mt}{{\widetilde M}}
\newcommand{\Llt}{{\widetilde{\mathcal L}}}
\newcommand{\yt}{{\tilde y}}
\newcommand{\vt}{{\tilde v}}
\newcommand{\Ppt}{{\widetilde{\mathcal P}}}
\newcommand{\bp }{{\bar \partial}} 

\newcommand{\Remark}{{\it Remark}}
\newcommand{\Proof}{{\it Proof}}
\newcommand{\ad}{{\rm ad}}
\newcommand{\Om}{{\Omega}}
\newcommand{\om}{{\omega}}
\newcommand{\eps}{{\varepsilon}}
\newcommand{\Di}{{\rm Diff}}
\newcommand{\pr}{{\mathrm{pr}}}

\renewcommand{\a}{{\mathfrak a}}
\renewcommand{\b}{{\mathfrak b}}
\renewcommand{\k}{{\mathfrak k}}
\renewcommand{\l}{{\mathfrak l}}
\renewcommand{\o}{{\mathfrak o}}
\renewcommand{\so}{{\mathfrak so}}
\renewcommand{\u}{{\mathfrak u}}
\renewcommand{\t}{{\mathfrak t }}
\newcommand{\Cinf}{C^{\infty}}
\newcommand{\la}{\langle}
\newcommand{\ra}{\rangle}
\newcommand{\half}{\scriptstyle\frac{1}{2}}
\newcommand{\p}{{\partial}}
\newcommand{\notsub}{\not\subset}
\newcommand{\iI}{{I}}               
\newcommand{\bI}{{\partial I}}      
\newcommand{\LRA}{\Longrightarrow}
\newcommand{\LLA}{\Longleftarrow}
\newcommand{\lra}{\longrightarrow}
\newcommand{\LLR}{\Longleftrightarrow}
\newcommand{\lla}{\longleftarrow}
\newcommand{\INTO}{\hookrightarrow}

\newcommand{\QED}{\hfill$\Box$\medskip}
\newcommand{\UuU}{\Upsilon _{\delta}(H_0) \times \Uu _{\delta} (J_0)}
\newcommand{\bm}{\boldmath}

\title[Deformation of calibrated submanifolds]{\large Strongly homotopy Lie algebras and deformations of calibrated submanifolds}

\author[D. Fiorenza]{Domenico Fiorenza}
\address{Dipartimento di Matematica ``Guido Castelnuovo",
Universit\`a di Roma ``La Sapienza",
P.le Aldo Moro 5, I-00185 Roma Italy}
\email{fiorenza@mat.uniroma1.it}

\author[H. V. L\^e]{H\^ong V\^an L\^e}
\address{Institute of Mathematics, Czech Academy of Sciences, Zitna 25, 11567 Praha 1, Czech Republic}
\email{hvle@math.cas.cz}

\author[L. Schwachh\"ofer]{Lorenz Schwachh\"ofer}
\address{Faculty for Mathematics,
TU Dortmund University,
Vogelpothsweg 87, 44221 Dortmund, Germany}
\email{lschwach@math.tu-dortmund.de} 

\author[L. Vitagliano]{Luca Vitagliano}
\address{DipMat, Universit\`a degli Studi di Salerno \& Istituto Nazionale di Fisica Nucleare, GC Salerno, via Giovanni Paolo II n${}^\circ$ 123, 84084 Fisciano (SA) Italy.}
\email{lvitagliano@unisa.it}

\thanks{Research of HVL  was supported by RVO:67985840 and the GA\v CR-project 18-00496S. LS acknowledges partial support by grant SCHW893/5-1 of the Deutsche Forschungsgemeinschaft.}

\date{\today}

\abstract  For an element $\Psi$  in  the graded vector space $\Om^*(M, TM)$
of tangent bundle valued forms on a smooth manifold $M$, a $\Psi$-submanifold is defined as a submanifold $N$ of $M$ such that $\Psi_{|N} \in \Om^*(N, TN)$.
The class of  $\Psi$-submanifolds  encompasses   calibrated  submanifolds,   complex submanifolds and  all Lie subgroups  in compact Lie groups.  The graded vector space $\Om^*(M, TM)$ carries a natural graded Lie algebra structure, given by the  Fr\"olicher-Nijenhuis bracket $[-,- ]^{FN}$. When $\Psi$ is an odd degree element with $[ \Psi, \Psi]^{FN} =0$, we associate to a  $\Psi$-submanifold  $N$ a  strongly homotopy  Lie algebra, which governs  the  formal and (under certain assumptions) smooth deformations of  $N$ as a $\Psi$-submanifold, and we show that under certain assumptions these deformations form an analytic variety. As an application we revisit  formal and smooth deformation  theory of complex  closed submanifolds  and of $\varphi$-calibrated  closed submanifolds, where $\varphi$ is a parallel  form in a real analytic Riemannian manifold. 
\endabstract


\maketitle
\tableofcontents

\section{Introduction}

Let $(M, g)$ be a Riemannian   manifold and $\nabla$  the Levi-Civita connection.
A  differential form $\varphi$  is  called  {\it parallel}  if
$\nabla  \varphi = 0$.   In this case we shall write  $(M, g, \varphi)$. When we want to stress that $\varphi$ has degree $p$ we shall write $\varphi^{p}$ for $\varphi$. 
    If  $\varphi$ is parallel,  $\varphi$ is closed and  its comass is constant.  Normalizing the comass  of $\varphi$,  we   regard  $\varphi$ as a  calibration.  All important  calibrated submanifolds  are $\varphi$-calibrated submanifolds   for some parallel  differential form $\varphi$   \cite{Dao1977, HL1982, Le1990, McLean1998, Joyce2007}\footnote{In \cite{Dao1977} Dao, based on   previous  work by Federer  and Lawson,   proposed  to use   parallel differential forms  to study area-minimizing    real currents, but he did not invent  the  word ``calibration".}. On the other hand,   $\varphi$-calibrated submanifolds  play an important r\^ole in the geometry  of  manifolds with special holonomy, in higher  dimensional gauge theory and in string theory as  ``super-symmetric cycles"  or ``branes"  \cite{SYZ1996, DT1998,Tian2000,  GYZ2003, AW2003,   Joyce2007,  DS2011,  Walpuski2012, Walpuski2014}.  Note that  manifolds with special holonomy   always admit   parallel  forms, see  Subsection \ref{subs:parallel} below. 

Deformation theory  of closed calibrated  submanifolds  has been initiated  by McLean \cite{McLean1998} inspired by   similarities between  calibrated submanifolds and complex submanifolds.  McLean  considered deformations of   special  Lagrangian, associative,   coassociative  and Cayley submanifolds.
 In \cite{LV2017}  L\^e-Van\v zura  observed that  any $\varphi$-calibrated submanifold $L$ 
  in  a Riemannian manifold $(M, g)$ considered  by McLean (as well  as  any K\"ahler  submanifold) satisfies the following  Harvey-Lawson identity \cite[Definition 1.1]{LV2017}
\begin{equation}
|\varphi (\xi)|^2  + | \Psi_E (\xi)| ^2  = |\xi|  ^2    \text{  for  all }   x\in M,\label{eq:HL}
\end{equation}
with $\xi$ an element in the Grassmannian of unit  decomposable $k$-vectors  in $T_x M$, for  some    $E$-valued   form $\Psi_E \in \Om^*(M,E)$, where $E$  is a Riemannian  vector bundle over $M$.
In this case  the  defining equation of  $\varphi$-calibrated submanifolds $L$ is  equivalent  to $(\Psi_E)_{| L} = 0$.  McLean showed  that, in the reformulation of \cite{LV2017} using  the  Harvey-Lawson identity, 
 the equation  $(\Psi_E)_{|L} = 0$ is essentially elliptic   for special Lagrangian   and coassociative submanifolds,  and using the standard  elliptic theory
 he proved that deformations of those submanifolds are unobstructed.  Additionally,  he proved that  the equation $(\Psi_E)_{|L} =0$ is  elliptic for   associative and Cayley submanifolds   $L$,   but   deformation of those  calibrated submanifolds may be obstructed.
 
 Further works on deformations of calibrated submanifolds    are devoted  to  the smoothness   and  the Zariski  tangent space to  the   moduli space      of closed   calibrated submanifolds   that are  special Lagrangian, associative, coassociative  and Cayley in  (tamed) almost/nearly Calabi-Yau, $G_2$ and Spin(7)-manifolds  \cite{AS2008, AS2008b, GIP2003, Gayet2014, Kawai2017, Ohst2014}, or to  similar questions concerning   calibrated submanifolds with  elliptic boun\-da\-ry  condition \cite{Butscher2003,  KL2009, GW2011, Ohst2014}   and   non-com\-pact  calibrated  submanifolds  of  certain type \cite{JS2005, KL2012,  Lotay2009}.

  In the present paper  we propose  a  new   approach  to deformation of calibrated   submanifolds.
  Firstly,  we   do not look for  a Harvey-Lawson type identity. 
  Instead, using   the  first cousin principle  we characterize $\varphi$-calibrated  submanifolds up to first order
  via the vector-valued form  $\hat \varphi \in \Om^*(M, TM)$ that is obtained  from $\varphi$ by   contraction  with the metric (Lemma \ref{lem:cal1}, see also Remark \ref{rem:prn}).  
    Motivated by Lemma \ref{lem:cal1}, 
     we  introduce  the notion of a $\Psi$-submanifold (Definition \ref{def:psisubm}) and develop  a general  deformation theory for closed $\Psi$-submanifolds for    any square-zero element $\Psi$ of odd degree in the graded Lie algebra $\Om^*(M, TM)$,  using  strongly homotopy Lie algebras 
       (Proposition \ref{prop:mc}).  This  generalizes the  assignment of a strongly homotopy  Lie algebra  to a  complex submanifold (Remark \ref{rem:compl}).
  In particular, we show that under some natural assumptions, the deformation space of $\Psi$-submanifolds is a finite dimensional analytic space (Theorem \ref{thm:Main-Psi}).

   Applying this to a parallel calibration, we prove that the moduli space of  $\hat \varphi$-submanifolds within a given $\varphi$-transversal homology class is an analytic space and hence, both  the formal and the smooth  deformation   problem  for closed $\varphi$-calibrated submanifolds
in $(M, g, \varphi)$
 are encoded  in its  associated   $L_\infty$-algebra (Theorem \ref{thm:main}).

 This paper is organized as follows.  In Section \ref{sec:pre}, we collect known results   concerning    parallel   differential forms  and the  Fr\"olicher-Nijenhuis bracket that are important  for  the main part of the paper.   In  Section \ref{sec:calibrated}, we introduce the  notion of a $\Psi$-submanifold (Definition \ref{def:psisubm})  which seems a good  notion   to  understand    deformations of 
 calibrated submanifolds (Corollary \ref{cor:caldeform}).  In Section \ref{sec:linfty}, we  assign to each $\Psi$-submanifold  a  canonical strongly homotopy Lie algebra,  if 
 $\Psi$ is a square-zero  element   of odd degree  in the  graded Lie algebra $(\Om^*(M, TM), [-,- ]^{FN})$ (Theorem \ref{thm:linfty}). In Section \ref{sec:def}  we define the deformation problem
 for $\Psi$-submanifolds and study formal deformations using this strongly homotopy  Lie algebra (Proposition \ref{prop:mc}). Moreover, we show that under certain conditions the deformation space is an analytic variety (Theorem \ref{thm:Main-Psi}). In Section \ref{subs:calibr}, we apply these results to study infinitesimal, smooth  and formal deformations of calibrated submanifolds in detail 
 (Proposition \ref{prop:infinitesimal}, Theorem \ref{thm:main}) and revisit the deformation theory of complex submanifolds (Theorem
 \ref{thm:c}, Remark \ref{rem:compl}). 
  
\medskip

  {\it  Notations and conventions}.

  $\bullet$  In this paper,  manifolds and their submanifolds  are denoted by  capital Latin letters $M, L$, etc.   When we want to   emphasize  the  dimension of    a manifold $M$ (resp. a submanifold $L$)  we
  write  $M^m$ (resp. $L^l$). The tangent map to a smooth map $f : M \to N$ is denoted by $T f : TM \to TN$, and its value at the point $x \in M$ by $T_x f : T_x M \to T_{f(x)} N$.
  
  $\bullet$  Small  Greek letters  usually denote
  scalar  valued forms  and capital  Greek letters  denote   vector  valued forms. Also, we use the Einstein summation convention summing over repeated indices whenever convenient.
  
  $\bullet$ For  a  scalar  valued  form $\varphi$  on $M$ we denote  by $\hat \varphi$ the associated $TM$-valued   form on $M$ obtained from $\varphi$ by contraction with the metric (see (\ref{eq:def-partial})  and the    sentence that follows for  explanation). 
  
  $\bullet$  For a  finite dimensional (resp. infinite dimensional) vector  space $V$  we denote by $0\in V$  (resp. $\mathbf{0}\in V$) the  origin  of $V$.
  
 $\bullet$ We adopt Getzler's conventions about $L_\infty$-algebras \cite{Getzler2009}.
 
 \section{Preliminaries}\label{sec:pre}
 \subsection{Parallel  differential forms on a  Riemannian  manifold}\label{subs:parallel}
 In this section, we  recall the classification of parallel   differential forms 
 on a  Riemannian manifold  $(M, g)$,  described  in   Tables 1,  2, 3, 4   from
\cite[Chapter 10]{Besse1987}.  

Let $\varphi$ be a  parallel  form on $(M, g)$  such that $\varphi$ is  not a  multiple of the volume form. Then  the restricted holonomy group $Hol^0 (M, g)$  is  contained  in the stabilizer
$Stab (\varphi)$   and therefore is strictly smaller than the    group $O(m)$.  Since  locally 
a  Riemannian manifold $(M, g)$ is a  product  of     Riemannian  manifolds  whose holonomy group action on  the tangent   bundle is irreducible, the classification
of  parallel  forms  on $(M, g)$ is reduced to the case  of   irreducible Riemannian manifolds $(M, g)$.  Symmetric Riemannian  spaces are examples
of    manifolds    admitting parallel  forms.

$\bullet$ The algebra  of parallel   forms on an  irreducible  symmetric space $M = G/H$  is  isomorphic to the algebra  of $Ad_H$-invariant  forms on $T_e G/H$.
In particular, if $M=G/H$ is compact,  then  the  algebra    of parallel forms  is isomorphic to 
the   de Rham cohomology algebra $H^*(M, \R)$. A list of the Poincar\'e polynomials of all the simply connected compact irreducible symmetric spaces
has been compiled by Takeuchi in \cite{Takeuchi1962}.

In 1955, Marcel Berger proved that if $(M, g)$ is a   simply-connected Riemannian manifold with irreducible holonomy group and nonsymmetric, then 
$Hol^0(M, g)$ must be one of
$SO(n)$, $U(m)$ (K\"ahler  manifolds) , $SU(m)$ (Ricci flat K\"ahler  manifolds, in particular Calabi-Yau manifolds), $Sp(m)$ (hyper-K\"ahler  manifolds), $Sp(m)\times  Sp(1)$ (quaternionic  K\"ahler  manifolds), $G_2$ ($G_2$-manifolds) or Spin(7)  (Spin(7)-manifolds). 

$\bullet$  The  algebra  of  parallel forms
on a K\"ahler manifold is generated  by   the K\"ahler   2-form $\om$.

$\bullet$  The algebra  of parallel  forms  on  a Ricci flat K\"ahler  manifold   is generated by the K\"ahler  2-form $\om$ and the real   and imaginary parts  $\Re \vol_\C$, $\Im  \vol _\C$ of the complex volume form. The latter  are called  special  Lagrangian forms, abbreviated as SL-forms.

$\bullet$ The algebra  of parallel  forms on a  quaternionic  K\"ahler manifold is generated by the quarternionic 4-form $\psi$.

$\bullet$ The algebra  of parallel  forms on a  hyper-K\"ahler manifold is generated by the three K\"ahler  2-forms.

$\bullet$ The algebra  of parallel  forms on a $G_2$-manifold is generated by the associative 3-form $\varphi$  and its dual coassociative 4-form  $*\varphi$.

$\bullet$  The algebra of parallel forms on  a Spin(7)-manifold  is generated by the  self-dual Cayley 4-form $\kappa$.

We also refer   the   reader  to  \cite{Bryant1987,Salamon1989} for  the  geometry of parallel  forms
on manifolds with special holonomy.

\subsection{Fr\"olicher-Nijenhuis bracket}\label{subs:fn}
  
Let us  recall  the  definition of the  Fr\"o\-li\-cher-Nijenhuis bracket
on $\Om^*(M, TM)$ following \cite[\S 8]{KMS1993}, see also \cite[\S 2.1]{KLS2017a} for a short account.

The space $Der(\Om^*(M))$ of  graded  derivations of the  graded commutative algebra $\Om^*(M)$  is a graded Lie algebra. First we recall  the definition of   \emph{algebraic  graded  derivations}
in $Der(\Om^*(M))$. They are  defined by   insertions $\iota_K$   for  $K \in \Om^*(M, TM)$. 
For $K = \alpha^l \otimes X$  we  define
   $\iota_K  \in Der(\Om^*(M))$   as follows   
\[
\iota_{\alpha^l \otimes X} \beta^r := \alpha^l \wedge (\iota_X \beta^r) \in \Om^{l+r-1}(M).
\]

Next we define  the linear   map 
$$\Ll: \Om^*(M, TM) \to Der(\Om^*(M)), \: K \mapsto \Ll_K,$$
\begin{equation}
\Ll _K : = [\iota_K, d] \in Der (\Om^*(M)). \label{eq:derFN}
\end{equation}

\begin{proposition}
	\label{prop:decomp1}(\cite[Theorem 8.3, p. 69]{KMS1993})
 	For any  graded derivation $D \in Der (\Om^*(M))$ there are  unique $K \in \Om^* (M, TM)$  and  $K' \in \Om^*(M, TM)$ such that 
 	$$ D = \Ll _K + \iota_{K'}.$$
	We have $K' =0$ if and only if  $[D, d] = 0$  and $D$ is algebraic if and only if $K = 0$.
\end{proposition}
  
 It follows from Proposition  \ref{prop:decomp1}  that  the map $\Ll$ is injective
and its  image  $\Ll (\Om ^*(M, TM))$  is  the centralizer  of $d$ in
$Der (\Om^*(M))$:
\begin{equation}\label{eq:commute1}
\Ll(\Om^* (M, TM)) = \{  D \in Der (\Om^*(M))|\:  [D, d] = 0\}.
\end{equation}
Hence,  $\Ll(\Om^* (M, TM))$ is  closed  under  the graded  Lie bracket $[-,- ]$  on $Der(\Om^*(M))$.
Then   we define {\it  the Fr\"olicher-Nijenhuis bracket}  $[-,- ] ^{FN}$ on $\Om^*(M, TM)$   as  the pull-back  of the   graded Lie bracket  on $Der(\Om^*(M))$ via  the  linear embedding $\Ll$, i.e., 
\begin{equation}
 \Ll_{[K, L]^{FN}}: = [\Ll_K, \Ll_L].  
 \label{eq:derFN2}
\end{equation}
Thus,  the Fr\"olicher-Nijenhuis bracket  provides   $\Om^*(M, TM)$  with the structure of a  $\Z$-graded (hence $\Z_2$-graded)
Lie algebra.

Furthermore  the Fr\"olicher-Nijenhuis  bracket enjoys the following functoriality with respect to    local diffeomorphisms. First of all, for a local diffeomorphism $f: M \to N$   and any $K \in  \Om^*(N, TN)$,  the pull-back  of $K$ by $f$ is defined as follows:
\begin{equation}\label{eq:pullback}
 (f^*K)_x (X_1, \cdots, X_l): = (T_xf)^{-1}  K_{ f(x)}(T_xf \cdot X_1, \cdots,  T_x f \cdot X_l).
\end{equation}
Then we have \cite[8.16, p. 74]{KMS1993}
\begin{equation}\label{eq:pullback2}
f^*[K, L] ^{FN} = [f^* K, f^*L]^{FN}.
\end{equation}

Let $(M, g)$  be a Riemannian  manifold.  Recall that     the  contraction $\p_g: \Lambda^l T^*M \longrightarrow \Lambda^{l-1} T^*M \otimes TM $
is defined  pointwise as follows \cite[(2.5)]{KLS2017a} 
\begin{equation} \label{eq:def-partial}
\p_g(\varphi^l) := (\iota_{e_i} \varphi^l) \otimes (e_i),
\end{equation}
where the sum is taken  pointwise over some orthonormal basis $(e_i)$ of $T_xM$.

We also   abbreviate  $\p_g (\varphi) $  by $\hat \varphi$.

\begin{remark}\label{rem:avp} To express  $\hat \varphi$, we  shall also  use the more   convenient  notion of   a {\it $(l-1)$-fold alternating vector product  $\rho_\varphi \in \Om^{(l-1)}(M, TM)$}
{\it defined  by a differential  form $\varphi \in \Om^l(M)$},
 \cite[(3.1)]{Robles2012}  such that 
\begin{equation}\label{eq:avp}
\varphi^l (u, v_2, \cdots, v_l)= g( u, \rho_{\varphi} (v_2, \cdots,  v_l)).
\end{equation}
Then   we   have  $\hat \varphi = \rho_\varphi$.
The  notion of  an  alternating  vector  product, introduced by Robles,   is a  generalization of the  notion of  a  {\it multi-linear   vector  cross  product}, introduced by  Gray \cite{Gray1969}, where   Gray
 imposed  a further  compatibility  between  a vector cross product, which is an alternating vector product,   on  a  pseudo-Riemannian  manifold $(M, g)$  and the  pseudo-Riemannian metric   $g$. Vector  cross products  have been  used  intensively  by Harvey-Lawson in their  study  of   calibrated geometry \cite{HL1982}.
\end{remark}

A straightforward computation via geodesic normal coordinates  yields  the following

\begin{proposition}\label{prop:parallel} (cf. \cite[Proposition 2.2]{KLS2017a}) For any parallel differential form $\varphi$ on a Riemannian manifold $(M, g)$  we have
$[\hat \varphi, \hat \varphi]^{FN} = 0$.
\end{proposition}

\begin{definition}\label{def:MC} We   say that   an element $\Psi \in \Om^{2l+1}(M, TM)$ is {\it of   square-zero},   if $[\Psi, \Psi]^{FN} = 0$.
\end{definition}

Observe that $[\Psi, \Psi]^{FN} = 0$ for any  $\Psi \in \Om^{2l}(M, TM)$ as $[-,- ] ^{FN}$ is graded skew-symmetric.

\section{$\varphi$-calibrated submanifolds and  $\Psi$-submanifolds}\label{sec:calibrated}

  In this section, motivated by the  geometry  of  calibrated  submanifolds (Lemma \ref{lem:cal1}),  we introduce the notion
of a $\Psi$-submanifold for any $ \Psi \in  \Om^*(M, TM)$  (Definition \ref{def:psisubm}). We relate  Lemma \ref{lem:cal1} and the notion of a  $\Psi$-submanifold  with  previous  work  on  calibrated  submanifolds, their further   extensions and related   results (Remarks \ref{rem:prn}, \ref{rem:psi}).  We   provide examples  of $\Psi$-submanifolds  that are not calibrated submanifolds (Example \ref{ex:psisubm}), including all
complex submanifolds  as well as  all   Lie subgroups  in  compact  Lie  groups.    Finally  we 
give a shorter  proof  of Robles'  result that if
 $\varphi^l$ is  a  parallel $l$-form  on a Riemannian manifold $(M, g)$  then  $\hat \varphi^l$-submanifolds $L$ are minimal submanifolds  if  the restriction
of $\varphi^l$ to $L$   does not vanish (Theorem \ref{thm:crit}),  which  shall be needed  in a later section.

For a submanifold $L$ in a manifold $M$ we denote by  $NL=TM_{|L}/TN$ the normal bundle of $L$ and  by $\pr: TM_{|L} \to NL$     the  canonical projection.
If $M$ is endowed  with a Riemannian  metric $g$ then we also identify
$NL$ with the  (Riemannian)  normal bundle  of $L$ that  is the   orthogonal complement to the tangent  bundle $TL$. 

\begin{lemma}\label{lem:cal1}  Let $\varphi^l$ be a calibration on a Riemannian manifold
	$(M, g)$  and $L$   a $\varphi^l$-calibrated  submanifold.  Then $\pr\circ \hat \varphi^l _{| L} = 0 \in \Om^*(L, NL)$.
\end{lemma}

\begin{proof} Let $L$ be a  $\varphi^l$-calibrated  submanifold.  Let $(e_i)$ and $(f_j)$ 
	be orthonormal bases of $T_xL$ and $N_xL$, respectively.
	Then  for any $i \in [1,k]$ we have
	\begin{equation}\label{eq:cal1}
	\pr \circ \hat \varphi ( e_1\wedge \cdots \widehat{e_i} \cdots \wedge  e_l) = \sum_{j =1}^{n-l} \varphi ( f_j\wedge e_1 \wedge \cdots \widehat{e_i}  \cdots \wedge  e_l) \otimes    f_j,
	\end{equation}
where $\widehat{e_i}$ stands for omission.
	By the first cousin principle for calibrated submanifolds \cite[p. 78]{HL1982} the  right hand side  of (\ref{eq:cal1}) vanishes. This completes  the proof.
\end{proof}

\begin{remark}\label{rem:prn} (1) Let  us denote by  $G_l(T_xM)$  the Grassmannian  of all unit  decomposable  $l$-vectors  in  $T_xM$   and by $\vec{T_xL}$  the unit
	$l$-vector   associated to  the oriented   tangent space $T_xL$ whose orientation is  defined by the  volume form $\varphi^l _{| L}$.  The Grassmanian  $G_l(T_xM)$  has   the natural Riemannian metric induced from the    Riemannian metric on 
	$T_xM$.  Note   that  the tangent  space $T_{\vec{T_xL}} G_l (T_xM)$  has  an orthogonal basis   consisting  of    $l$-vectors  of the form
	$f_j \wedge  e_1 \wedge \cdots \widehat{e_i} \cdots \wedge  e_l$.  Let $\tilde \varphi^l (x)$ denote the restriction of
	$\varphi^l(x)$ to  $G_l (T_xM)$.   Then  we have
	\begin{equation}\label{eq:prn}
	\la \pr \circ \varphi^l (e_1\wedge \cdots \widehat{e_i} \cdots \wedge  e_l),  f_j\ra  = \la  d_{ e_1\wedge \cdots \wedge e_l}\tilde \varphi^l(x), f_j \wedge  e_1 \wedge \cdots \widehat{e_i}  \cdots \wedge  e_l\ra
	\end{equation}
	where  the pairing in the LHS of (\ref{eq:prn}) is defined  via the Riemannian metric and where again, $\widehat{e_i}$ stands for omission.  Thus  the condition  $\pr\circ \hat \varphi^l _{| L} = 0 \in \Om^*(L, NL)$  is equivalent  to  the 
	condition  that $L$ is a  $\varphi$-critical manifold,   i.e., $\vec{T_xL}$ is a critical  point  of $\varphi_x$ for all $x \in L$;  see also
	\cite{Le1990}, \cite{HL2009}, \cite{Robles2012} for  study of
	$\varphi$-critical   submanifolds.  
	
	(2)  In \cite[Proposition 3.4]{Robles2012} Robles  gives a nice
	characterization  of  $\varphi$-critical   submanifolds  $L$  in  terms  of the   alternating vector product $\rho_\varphi$, namely  $TL$  is  $\rho_\varphi$-closed, i.e., $\rho_{\varphi|L} \in \Om ^* (L, TL)$.
	
	(3) Although  the    equality $\mathrm{pr} \circ  \hat \varphi _{|L}=0$    is equivalent to the condition   that $L$ is $\rho_\varphi$-closed, we   prefer  the expression $\mathrm{pr} \circ  \hat \varphi _{|L}=0$ , since it   says that $L$  is    the   zero set   of an $NL$-valued  differential form, what we shall  use   in our  deformation theory in later sections. Moreover,   this expression is similar  to that appearing in the deformation  theory for coisotropic submanifolds  in Poisson and Jacobi manifolds (see, e.g., \cite{LOTV2014}, and references therein), which led us   to our  search  for  $L_\infty$-algebras  governing   deformations of  calibrated   submanifolds.
\end{remark}

Lemma \ref{lem:cal1} motivates the following

\begin{definition}\label{def:psisubm}  Let   $M$ be a smooth manifold and  $\Psi\in \Om^{l}(M, TM)$.  A submanifold  $L^r \subset M$, where
	$r \ge l$,  will be called {\it a $\Psi$-submanifold},
	if  $\pr \circ \Psi_{|  L}  = 0 \in \Om^{l}(L^r, NL^r)$ or, equivalently,
	$ \Psi _{|L} \in \Om^l(L^r, TL^r)$.
\end{definition}

\begin{remark}\label{rem:psi}  The class   of 	 $\Psi$-submanifolds  of a manifold $M$ is larger  than   the class  of $\varphi$-critical  submanifolds, since   this definition  does not require  a metric on $M$. For instance,  any almost complex   submanifold  in an almost complex    manifold is a $\Psi$-submanifold. In contrast, the notion of a  $\varphi$-critical submanifold  in $M$  implicitly requires   a   Riemmanian metric   $g$  on $M$, which allows  us to associate  a tangent  space $T_xL$  of
	an oriented submanifold  $L \subset    M$  with  the  unit   $l$-vector  $\vec{T_x L}$.  
\end{remark}

Let us   recall  the  following  result of Robles, which we shall need  later.

\begin{theorem}\label{thm:crit} (\cite[Theorem 1.2]{Robles2012})
 Assume that $\varphi^l$  is a parallel form
on a Riemannian manifold $(M, g)$. Then
 a $\hat\varphi ^l$-submanifold $L$  is a minimal submanifold if $\varphi^l_{|  L} \not = 0$. 
 \end{theorem}

We provide below a  short  proof  of Theorem \ref{thm:crit}, using the argument in the proof of  Lemma 1.1 in \cite{Le1990}.

\begin{proof}   Let  $L$ be a  $\hat \varphi^l$-submanifold  in $(M, g)$. We shall compute the mean curvature $H$  of $L$. Let $V_1, \cdots, V_l$ be   local vector  fields
on $L$ such that  $| V_1 \wedge \cdots \wedge  V_l| =1$. By Remark \ref{rem:prn}(1), for each $x \in  L$ the unit $l$-vector $\vec{T_x L}$  is a critical  point  of the  function  $\xi \mapsto \varphi^l(\xi)$. 
Hence we have
\begin{equation}\label{eq:const}
 \varphi^l(\vec{T_xL} ) = c
 \end{equation}
 for some  constant $c$  (this follows  from  the classification of  parallel  forms on $(M, g)$, see e.g. 
\cite[Theorem 10.108, Corollary 10.110]{Besse1987}  and  Subsection \ref{subs:parallel}).  Recall that $c \not = 0$ by the assumption of Theorem \ref{thm:crit}.

Let  $X$   be a normal vector on $L$. Using the argument in the proof of  Lemma 1.1 in \cite{Le1990} we compute
\begin{eqnarray}
 0 = (\iota_X d\varphi^l)(V_1, \cdots, V_l)  = \sum_{i=1}^l (-1)^iV_i \bigl(\varphi^l (X, V_1, \cdots, \hat{V_i}, \cdots, V_l)\bigr) \nonumber\\
 - X \bigl(\varphi^l (V_1, \cdots, V_l)\bigr) + \sum_{1\le i < j \le l} (-1) ^{i+j} \varphi^l ([V_i, V_j], X, \cdots, \hat {V_i}, \cdots, \hat {V_j}, \cdots,  V_l)\nonumber \\
 + \sum_{i=1}^l (-1) ^i \varphi^l ([X, V_i], \cdots \hat{V_i}  , \cdots  V_l),\label{eq:const1}
 \end{eqnarray}
where $\hat V_i$ stands for omission of this entry. The first  and third  terms in (\ref{eq:const1}) are zero, since $L$ is a $\hat \varphi^l$-submanifold.
 The second term  is zero by (\ref{eq:const}).  Hence we get  from (\ref{eq:const1})
\[
0 = \sum_{ i=1}^l (-1) ^i \varphi^l ([X, V_i], \cdots, \hat{V_i} , \cdots, V_l)  = c\sum_{i=1}^l \la [X, V_i ], V_i \ra = c \la -H, X \ra .
\]
Since $c \not =0$  we obtain $ H =0$.   This  proves Theorem \ref{thm:crit}.
\end{proof}

\begin{corollary}\label{cor:caldeform}  Let $\varphi$ be a  parallel calibration. A deformation of a closed $\varphi$-calibrated   submanifold inside  the  class of $\hat \varphi$-submanifolds 
remains  in the  subclass of $\varphi$-calibrated submanifolds.
\end{corollary}
\begin{proof}  Let $\varphi$ be a  calibration and let $L_t$, $ t\in [0,1] $, be  a continuous  family of closed  $\hat \varphi$-submanifolds  such that $L_0$ is   a  $\varphi$-calibrated  submanifold.  Then  $\varphi_{|L_0} \not = 0$  and therefore  $\varphi_{|L_t} \not = 0$ for    sufficiently small  $t$.  By   Theorem \ref{thm:crit}, $L_t$  is also  a minimal submanifold  for such small $t$, in particular  the volume of $L_t$  is  constant  around $t =0$. Since	$L_0$ is a calibrated  submanifold, it follows that the $L_t$ are  calibrated  submanifolds  for  all sufficiently small  $t$.  Then the set  of  all values $t$   such that $L_t$ is a $\varphi$-calibrated submanifold  is  an open  subset  in the interval $[0,1]$.   On the other hand,   since  the volume function  is continuous,   the set  of  all values  $t$ such that $L_t$  is a $\varphi$-calibrated
	submanifold  is also closed.  This   completes the proof of Corollary \ref{cor:caldeform}.
\end{proof}

\begin{example}\label{ex:psisubm}

1. By Remark \ref{rem:prn} (1),  each $\varphi$-calibrated  submanifold  is a
$\hat \varphi$-submanifold. In particular,  every  associative submanifold
$L^3$ in  a $G_2$-manifold  $M^7$
is a $\hat \varphi$-submanifold, where $\varphi$ is the associative 3-form on $M^7$.  We claim that every  3-dimensional $\hat \varphi$-submanifold is an associative submanifold.  To prove  this assertion we   regard $\hat \varphi \in \Om^2 (M^7, TM^7)$ as the
2-fold cross  product  $TM^7 \times  TM^7 \to TM^7$:   $\varphi (X, Y, Z) = \la X \times Y, Z \ra$  where $\times $ denotes the cross product \cite{HL1982, KLS2017a}. Then our assertion follows from the first cousin principle for  $\hat \varphi$-submanifolds
and the   observation  that  a 3-plane is associative if
and only if  it is invariant under the 2-fold cross
product \cite{HL1982}.

2. Let us consider  a complex  manifold $(M, g, J)$. We regard $J$  as an element in $\Om^1(M, TM)$. Clearly
 a submanifold  $L$  in  $M$ is  a $J$-submanifold if and only if it is a complex  submanifold.
 
 3. Let $* \varphi$ be the  coassociative 4-form  on a  $G_2$-manifold
 $M^7$. The  associated form  $\widehat{* \varphi} \in  \Om^3 (M^7, TM^7)$  is often denoted by $\chi$ and   called the  3-fold cross  product \cite{HL1982, KLS2017a}.
 
 (a) It is shown  in the proof of Lemma  5.6  in \cite{KLS2018} that
 a  3-submanifold  $L^3 \subset M^7$  is a   $\chi$-submanifold if  and only if it is an associative submanifold.  
 
 (b) By Remark \ref{rem:prn} (1),
 every coassociative  submanifold $L^4$  is a $\chi$-sub\-ma\-ni\-fold.  We claim that  a   4-dimensional $\chi$-submanifold is
  a coassociative  submanifold.  To prove this it suffices to show that  the coassociative planes (up to orientation) are the only  critical points
  of the function  $\widetilde{*\varphi}$  defined in  Remark \ref{rem:prn} (1).   This assertion is equivalent  to   the statement that
  the associative  planes (up to orientation) are the only  critical  points of the  function $\tilde {\varphi}$, which has   been  proved in  Example \ref{ex:psisubm} 1.

  It is not hard to  conclude from  (a) and (b)  that a  $\chi$-submanifold  in   a $G_2$-manifold  is either an  associative or  a coassociative submanifold. Thus,  we regard  $\chi$  as an analogue  of the  complex  form $J \in \Om^1(M, TM)$ in  complex  geometry.
  In \cite{KLS2017a} Kawai-L\^e-Schwachh\"ofer   gave another interpretation  of this fact, proving   that  a $G_2$-structure is torsion-free if and only if   $[\chi, \chi]^{ FN}$ vanishes.
 
 4.  Let $\alpha : = Re (vol_\C)$  be  the SL-calibration on
 a Calabi-Yau manifold $(M, g, \om, vol_\C)$. Remark \ref{rem:prn} (1) implies
 that every special Lagrangian submanifold  $L\subset M$  is a $\hat \alpha$-submanifold.

  5. Let $M^7$ be a $G_2$-manifold and $\varphi$ the defining  associative  3-form.  In  \cite{LV2017} L\^e-Van\v zura define  a  form $\tau \in  \Om^4(M^7, TM^7)$  as follows.   For  $x, y, z, w \in TM^7$  we set (\cite[(1.17),  Theorem 1.18,  p. 117]{HL1982}, see  also  \cite[Remark 4.2]{LV2017})
\begin{equation}
 \tau(x, y, z, w): = -(\varphi(y, z, w) x+ \varphi(z, x, w) y+ \varphi(x, y, w) z+ \varphi(y,x,z) w).\label{eq:tau7}
\end{equation}
Then any $4$-submanifold  in $M^7$  is a $\tau$-submanifold.

6. Let $M^8$ be a  Spin(7)-manifold and $\psi^4$  its  defining Cayley  form.  Recall that   $\psi^4(X, Y, Z, W) = \la P(X, Y, Z), W\ra$ where
$P$ is the  3-fold vector cross product, see e.g. \cite{Fernandez1986}.
By Remark \ref{rem:prn}(1), every Cayley submanifold  is a $\hat \psi^4$-submanifold.  
Since  any $\hat \psi^4$-submanifold  $L$ is  invariant  under  the triple  product $P$, $L$ must be a Cayley  submanifold. 

7. Let  $G$ be a compact Lie group  provided with  the  Killing metric. Denote by  $\om^3$  the Cartan 3-form on $G$. The calibration $\om^3$ has been first considered by Dao in \cite{Dao1977} and later  by Tasaki \cite{Tasaki1985}. By Theorem 3.1 in \cite{Le1990} any  3-dimensional  Lie subgroup in
$G$ is a $\hat \om$-submanifold. Since     the  tangent  space $T_e G$ is  invariant under  the Lie bracket, any  Lie subgroup in  $G$ is  a $\hat \om$-submanifold.
  In   \cite[Section 3]{Le1990} L\^e classified  stably minimal  3-dimensional subgroups in compact  semi-simple  Lie groups of classical type, see also \cite{Le1990b} for the classification of all
  stably minimal simple Lie subgroups in classical Lie groups. Clearly,
non-stably  minimal Lie  subgroups   cannot be  calibrated submanifolds, since  calibrated submanifolds are  area-minimizing, and hence   stably minimal.

8. Let $\theta^3, \theta ^5, \cdots, \theta^{2m-1}$ be  bi-invariant forms on $SU(m)$. By Theorem  3.4 in \cite{Le1990} for any
$n< m$   the     standard  subgroup  $SU(n) \subset SU(m)$ is a $\hat\phi$-submanifold  for
$\phi = \theta^1 \wedge \cdots  \wedge \theta^{2n-1}$.
\end{example}

         \section{The $L_\infty$-algebra  associated  to a $\Psi$-submanifold}\label{sec:linfty}

We  use  the  notation  in  the previous  sections, in particular,  $NL$  denotes  the normal   bundle  of   a
	submanifold  $L$ in a manifold  $M$.
In this section, using Voronov's derived bracket construction \cite{Voronov2005},
we prove the following.

\begin{theorem}\label{thm:linfty}  Let  $\Psi \in \Om ^{*}(M, TM)$ be an odd degree element which is square-zero, i.e., such that  $[ \Psi, \Psi]^{FN} = 0$,
 and  let $L$  be a $ \Psi$-submanifold. Then the cochain complex $\Om^*(L,NL)[-1]$  carries a canonical  $\Z_2$-graded $L_\infty$-algebra structure. 
If $\deg\Psi =1$  then this $\Z_2$-graded  $L_\infty$-algebra is also a  $\Z$-graded
$L_\infty$-algebra.
\end{theorem} 

As a corollary, taking into account Example \ref{ex:psisubm} (1),(2), we obtain the following

\begin{corollary}\label{cor:linfty} 1.  Assume that  $\varphi^l$ is a parallel  $l$-form
on a   Riemannian manifold $(M, g)$  and  $L$  is a closed  $\varphi^l$-calibrated submanifold. If $l$ is even, then  there is a canonical  $\Z_2$-graded $L_\infty$-algebra  structure on $\Om^*(L,NL)[-1]$. If $l$ is odd, then there is a canonical  $\Z_2$-graded $L_\infty$-algebra  structure on $\Om^*(L \times S^1,N(L \times S^1))[-1]$.

2. (cf. \cite{Manetti2007})  For any  closed  complex  submanifold $L$  in a  complex manifold $M$  there  is 
a  canonical $\Z$-graded  $L_\infty$-algebra structure on $\Om^*(L,NL)[-1]$.

3.  For every closed associative 
submanifold $L^3$ in a $G_2$-manifold  $(M^7, \varphi^3)$ there are canonical $\Z_2$-graded $L_\infty$-algebra structures both on $\Om^*(L^3,NL^3)[-1]$ and on $\Om^*(L^3\times S^1,N(L^3\times S^1))[-1]$.
\end{corollary}
\begin{proof}Let $L$ be a   closed $\varphi^l$-calibrated submanifold of the  Riemannian manifold $(M, g)$.  If $l$ is even, then  $L$ is a $\hat \varphi^l$-submanifold  of $M$. If $l$ is odd, then 
$L \times S^1$ as  a  $\widehat {\varphi^l \wedge dt}$-submanifold  of $M \times  S^1$. This proves statement {\it 1}. Statement {\it 2} is immediate, as any  complex submanifold is a $J$-submanifold. Finally, any associative  submanifold  $L^3$  of a $G_2$-manifold  $(M^7, \varphi^3)$ is  a $\varphi^3$-calibrated submanifold, so we have an $L_\infty$-structure on $\Om^*(L^3\times S^1,N(L^3\times S^1))[-1]$ by statement {\it 1}. On the other hand,  it has been   showed in  \cite{KLS2018}  that   $L^3$ is $\widehat{*\varphi^3}$-submanifold  (see Example \ref{ex:psisubm}.3). This proves statement {\it 3}.
\end{proof}

The remainder  of this section  is devoted to the proof of  Theorem \ref{thm:linfty}.
First, let us recall  Voronov's construction of a $\Z_2$-graded $L_\infty$-algebra from a set  of   V-data. 
A set of {\it $V$-data}   is  a  quintuple $(\mathfrak g, \a, \mathrm{j}, P, \triangle)$, where 
\begin{itemize}
\item $\mathfrak g = \mathfrak g_0 \oplus \mathfrak g_1$ is a $\Z_2$-graded Lie algebra (with Lie bracket $[-,-]$),
\item $\a$ is  an abelian Lie algebra;
\item $\mathrm{j}: \a \to \mathfrak g$ is a Lie algebra inclusion;
\item $P : \mathfrak g \to \a$  is a (not necessarily bracket preserving) projection, inverting $\mathrm{j}$ from the left and such that $\ker P \subseteq \mathfrak g$ is a Lie subalgebra,
\item  $\triangle  \in (\ker P) \cap \mathfrak g_1$ is an element such that $[\triangle, \triangle] = 0$.
\end{itemize}

\begin{proposition}\label{prop:voronov}(\cite[Theorem 1, Corollary 1]{Voronov2005}). Let $(L, \a, \mathrm{j}, P, \triangle)$ be a set of V-data. Then $\a[-1]$ is a $\Z_2$-graded  $L_\infty$-algebra with multibrackets 
\begin{equation}\label{eq:multib}
 \mathfrak l_n(a_1, \cdots , a_n) = (-1)^\star P[\cdots [[\triangle, \mathrm{j} (a_1)], \mathrm{j}(a_2)],\cdots , \mathrm{j}(a_n)].
 \end{equation}
where 
\[
\star = (n-1)|a_1| + (n-2)|a_2| + \cdots + |a_{n-1}|  + \frac{n(n+1)}{2},
\]
and the vertical bars $|-|$ denote the degree in $\mathfrak a [-1]$.
\end{proposition}

Replacing $\Z_2$ by $\Z$ in the definition of $V$-data, Formula (\ref{eq:multib}) gives  a $\Z$-graded  $L_\infty$-algebra.
A homotopy Lie theoretic interpretation of Voronov's $L_\infty$-algebra structure on $\mathfrak{a}[-1]$ can be found in \cite{Bandiera2015}.

The proof of Theorem \ref{thm:linfty} will now go through several steps. The first step consists in associating V-data to a $\Psi$-submanifold $L$ equipped with a tubular neighborhood $\tau$.

If $j: L \hookrightarrow M$ is a submanifold, a {\em tubular neighborhood of $L$ in $M$} is defined to be a  diffeomorphism
$\tau :N_\epsilon L \to U \subset M$ from an open neighborhood $N_\epsilon L \subset NL$ of $\mathbf{0}$ onto an open neighborhood of $L$ in $M$ such that
$\tau \circ {\mathbf{0}} = j$, where $\mathbf{0}\colon L\to NL$ is the zero section. Clearly, such maps exist, e.g. using the normal exponential w.r.t. some Riemannian metric on $M$. Furthermore, since we may assume that $N_\epsilon L \subset NL$ is a disc bundle and hence bundle equivalent to all of $NL$, it follows that we may w.l.o.g. replace $N_\epsilon L$ by all of $NL$. That is, from now on we shall assume that a tubular neighborhood is a diffeomorphism $\tau: NL \xrightarrow{\sim} U\subset M$.

\begin{definition} \label{def-4.4} Let    
$ \Psi \in \Om^{*}(M, TM)$ be an odd degree element with $[\Psi, \Psi]^{FN} = 0$, let $j\colon L\hookrightarrow M$ be a $ \Psi$-submanifold and $\tau : NL \to U \subset M$ be a tubular neighborhood of $L$ in $M$. 
Denote by $\pi\colon NL\to L$ the projection. The 5-tuple $(\mathfrak g_{L},  \a_{L}, \mathrm{j}_{L}, P_{L},  \triangle_{L,\tau})$ is defined as follows:
\begin{itemize}
\item The graded Lie algebra $\mathfrak{g}_{L}$ is $\Omega^*(NL,TNL)$ with the FN bracket;
\item the abelian graded Lie algebra $\mathfrak{a}_{L}$ is the graded vector space $\Omega^*(L,NL)$ endowed with the zero bracket;
\item the graded vector space morphism $\mathrm{j}_{L}\colon \mathfrak{a}_{L}\to \mathfrak{g}_{L}$ is 
defined on decomposable elements as $\mathrm{j}_{L}(\omega\otimes X)=\pi^*(\omega)\otimes \hat{X}$, where $\hat{X}$ is the canonical vertical lift of $X$ given by the natural identification $N_{\pi(x)}L\cong \ker(\pi_*\colon T_xNL\to T_{\pi(x)}L)$;
\item the graded vector space morphism $P_{L}\colon \mathfrak{g}_{L}\to \mathfrak{a}_{L}$  is the composition 
\[
\Omega^*(NL,TNL)\xrightarrow{\vert_L} \Omega^*(L,TNL\bigr\vert_L) \xrightarrow{\mathrm{pr}} \Omega^*(L,NL),
\]
where the rightmost arrow $\mathrm{pr}$ is the natural projection induced  by the  projection $TNL_{|L} \to NL$, also denoted by $\mathrm{pr}$ in  Section \ref{sec:calibrated}, by identifying $L$ with a submanifold of $NL$ via the zero section $L\hookrightarrow NL$  (equivalently, $\mathrm{pr}$ is induced by the canonical splitting $TNL\vert_L=TL\oplus NL$);
\item the element $\Delta_{L,\tau}$ in $\mathfrak{g}_{L,\tau}$ is $\Delta_{L,\tau}=\tau^*\Psi$, where 
\[\tau^*\colon \Omega^*(M,TM)\to \Omega^*(NL,TNL)
\]
 is the pullback of tensors along the local diffeomorphism $\tau$.
\end{itemize}

\end{definition}

\begin{remark}Notice that, as the notation suggests, $\Delta_{L,\tau}$ is the only component of the 5-tuple $(\mathfrak g_{L},  \a_{L}, \mathrm{j}_{L}, P_{L},  \triangle_{L,\tau})$ which actually depends on the tubular neighborhood $\tau$.
\end{remark}

\begin{proposition}\label{induces-l-infinity-algebra-structure}
The 5-tuple $(\mathfrak g_{L},  \a_{L},\mathrm{j}_{L}, P_{L},  \triangle_{L,\tau})$ associated with a $\Psi$-manifold is a 5-tuple of V-data. As a consequence the graded vector space $\a_{L}[-1]=\Omega^*(L,NL)[-1]$ carries a $\mathbb{Z}_2$-graded $L_\infty$-algebra structure induced by this data. When $\Psi$ has degree 1, this is actually a $\mathbb{Z}$-graded $L_\infty$-algebra structure. 
\end{proposition}
\begin{proof} The map $\mathrm{j}_{L}$ is injective, the map $P_{L}$ is surjective, and one manifestly has $P_{L}\circ \mathrm{j}_{L}=\mathrm{id}_{\mathfrak{a}_{L}}$ so we are left with showing $[\mathrm{j}_{L}\mathfrak{a}_{L},\mathrm{j}_{L}\mathfrak{a}_{L}]=0$, that $\ker P_{L}$ is a Lie subalgebra of $\mathfrak g_{L}$, that $\triangle_{L,\tau}\in \ker P_{L}$ and $[\triangle_{L,\tau},\triangle_{L,\tau}]=0$. To this aim, consider the composition
\[
\tilde{P}_{L}=\mathrm{j}_{L}\circ P_{L}\colon \Omega^*(NL,TNL)\to \Omega^*(NL,TNL).
\]
It is shown in  \cite{KLS2018} that the image of $\tilde{P}_{L}$ is an abelian  subalgebra of the graded Lie algebra 
$(\Om^* (N L, TNL), [-, -] ^{FN})$ and that $\ker \tilde{P}_{L}$ is closed  under  the Fr\"olicher-Nijenhuis  bracket. As $P_{L}$ is surjective, the image of $\tilde{P}_{L}$ coincides with the image of $\mathrm{j}_{L}$, so that $\mathrm{j}_{L}(\mathfrak{a}_{L})$ is an abelian subalgebra of $\mathfrak g_{L}$. As $\mathrm{j}_{L}$ is injective, we have $\ker \tilde{P}_{L}=\ker P_{L}$, and so $\ker P_{L}$ is a Lie subalgebra of $\mathfrak g_{L}$.
By the naturality of the Fr\"olicher-Nijenhuis bracket, we have 
\begin{equation}
[\triangle_{L,\tau},\triangle_{L,\tau}]=[\tau^*\Psi,  \tau^*\Psi] ^{FN} = \tau ^* [\Psi, \Psi] ^{FN}=0.\label{eq:exp}
\end{equation}
Finally, as $L$ is a $\Psi$-manifold in $M$ and $\tau$ is a diffeomorphism relative to $L$ in a neighborhood of $L$ (identified with the zero section in $NL$), we have that $L$ is a $\triangle_{L,\tau}$-manifold in $NL$. Therefore, $P_{L}\triangle_{L,\tau}=0$  
by definition of $\triangle_{L,\tau}$-manifold.
\end{proof} 

The underlying graded vector space of the $L_\infty$-algebra structure induced on $\Omega^*(L,NL)[-1]$ by Proposition \ref{induces-l-infinity-algebra-structure} is independent of $\tau$. 
Our next step will consist in showing that also the $L_\infty$-algebra structure is actually independent of $\tau$, up to isomorphism. To begin with, let us show that a reparameterization of the tubular neighborhood leaves the $L_\infty$-algebra structure unchanged up to isomorphism.

\begin{lemma} Let $\tau_0$ and $\tau_1$ be two tubular neighborhoods of $L$ in $M$ such that $\tau_1=\tau_0\circ\psi$ for some diffeomorphism $\psi$ of $NL$ relative to $L$. Then $\psi$ induces an isomorphism of V-data between $(\mathfrak{g}_{L},\mathfrak{a}_{L},\mathrm{j}_{L},P_{L},\Delta_{L,\tau_0})$ and $(\mathfrak{g}_{L},\mathfrak{a}_{L},\mathrm{j}_{L},P_{L},\Delta_{L,\tau_1})$. In particular $(\mathfrak{g}_{L},\mathfrak{a}_{L},\mathrm{j}_{L},P_{L},\Delta_{L,\tau_0})$ and $(\mathfrak{g}_{L},\mathfrak{a}_{L},\mathrm{j}_{L},P_{L},\Delta_{L,\tau_1})$ induce isomorphic $L_\infty$-algebra structures on $\Omega^*(L,NL)[-1]$.
\end{lemma}

\begin{proof} As $\psi$ is a diffeomorphism of $NL$ relative to $L$ the pullback along $\psi$ induces  commutative diagrams
\[
\xymatrix{
\Omega^*(L,NL)\ar[r]^-{\mathrm{j}_{L}}\ar[d]^{\psi^*} &\Omega^*(NL,TNL)\ar[d]^{\psi^*}\\
\Omega^*(L,NL)\ar[r]^-{\mathrm{j}_{L}} &\Omega^*(NL,TNL)
}
\quad;\qquad
\xymatrix{
\Omega^*(NL,TNL)\ar[r]^-{P_{L}}\ar[d]^{\psi^*} &\Omega^*(L,NL)\ar[d]^{\psi^*}\\
\Omega^*(NL,TNL)\ar[r]^-{P_{L}} &\Omega^*(L,NL)
}
\]
Finally, we have 
\[
\psi^*\Delta_{L,\tau_0}=(\tau_0^{-1}\circ\tau_1)^*(\tau_0^*\Psi)=\tau_1^*\Psi=\Delta_{L,\tau_1}.
\]
\end{proof}

 In order to prove that the $L_\infty$-algebra structure $\Omega^*(L,NL)[-1]$ is generally independent of $\tau$, up to isomorphism, as we can not directly compare two distinct tubular neighborhoods of $L$ in $M$ it is convenient to pass to formal neigborhoods.

\begin{definition} Let    
$ \Psi \in \Om^{*}(M, TM)$ be an odd degree element with $[\Psi, \Psi]^{FN} = 0$, let $L\subset M$ be a $ \Psi$-submanifold and $\tau : NL \to U \subset M$ be a tubular neighborhood of $L$ in $M$.  Finally, let $NL_{\mathrm{for}}\hookrightarrow NL$ be the formal neighborhood 
of $L$ in $NL$ via the zero section embedding $s_0\colon L\hookrightarrow NL$. 
We recall that working in the formal neighborhood of $L$ means working only with \emph{$\infty$-jets} (of functions, sections, etc.) transverse to $L$ (see, e.g., \cite[Section 4.1]{CS2008}, were a similar situation is discussed in details). In the same notation as Proposition \ref{induces-l-infinity-algebra-structure}, the 5-tuple $(\mathfrak g_{L}^{\mathrm{for}},  \a_{L}^{\mathrm{for}}, P_{L}^{\mathrm{for}}, \mathrm{j}_{L}^{\mathrm{for}}, \triangle_{L,\tau}^{\mathrm{for}})$ is the restriction to $NL_{\mathrm{for}}$ of the 5-tuple $(\mathfrak g_{L},  \a_{L}, P_{L}, \mathrm{j}_{L}, \triangle_{L,\tau})$.
\end{definition}

\begin{remark}
Notice that the graded abelian Lie algebras $\a_{L}$ and $\a_{L}^{\mathrm{for}}$ actually coincide: they both are the graded vector space $\Omega^*(L,NL)$ endowed with the zero bracket. In particular the restriction to $NL_{\mathrm{for}}$ is the identity morphism on $\Omega^*(L,NL)$.
\end{remark}

\begin{proposition}\label{proposition-restriction}
The 5-tuple $(\mathfrak g_{L}^{\mathrm{for}},  \a_{L}^{\mathrm{for}}, \mathrm{j}_{L}^{\mathrm{for}}, P_{L}^{\mathrm{for}}, \triangle_{L,\tau}^{\mathrm{for}})$ is a set of V-data and so induces a $\mathbb{Z}_2$-graded $L_\infty$-algebra structure on $\a_{L}^{\mathrm{for}}[-1]=\Omega^*(L,NL)[-1]$. Moreover this  $\mathbb{Z}_2$-graded $L_\infty$-algebra structure coincides with that induced on $\Omega^*(L,NL)$ by the V-data $(\mathfrak g_{L},  \a_{L}, P_{L}, \mathrm{j}_{L}, \triangle_{L,\tau})$.
\end{proposition}

\begin{proof}
The proof follows analogous lines as those of \cite[Sections 4.1]{CS2008}, and we leave the obvious translation to the reader.
\end{proof}

\begin{lemma} \label{lemma-gauge-equivalent} Let $\tau_0$ and $\tau_1$ be two isotopic tubular neighborhoods of $L$ in $M$. Then $\triangle_{L,\tau_0}^{\mathrm{for}}$ and $\triangle_{L,\tau_1}^{\mathrm{for}}$ are gauge equivalent square-zero elements in $\mathfrak g_{L}^{\mathrm{for}}$.  In particular the V-data $(\mathfrak g_{L}^{\mathrm{for}},  \a_{L}^{\mathrm{for}}, \mathrm{j}_{L}^{\mathrm{for}}, P_{L}^{\mathrm{for}}, \triangle_{L,\tau_0}^{\mathrm{for}})$ and $(\mathfrak g_{L}^{\mathrm{for}},  \a_{L}^{\mathrm{for}},  \mathrm{j}_{L}^{\mathrm{for}}, P_{L}^{\mathrm{for}}, \triangle_{L,\tau_1}^{\mathrm{for}})$ induce isomorphic $L_\infty$-algebra structures on $\mathfrak{a}_L^{\mathrm{for}}[-1]=\Omega^*(L,NL)[-1]$.
\end{lemma}
\begin{proof}
By definition of isotopic tubular neighborhoods, there exist a smooth family $\Phi_t$ of maps $\Phi_t\colon NL\to M$, with $t\in [0,1]$, which are diffeomorphisms on their images and such that $\Phi_t\circ s_0=j$ for every $t\in [0,1]$, such that $\Phi_0=\tau_0$ and $\Phi_1=\tau_1$. Let $\hat{\Phi}_t$ be the composition of $\Phi_t$ with the embedding $NL_{\mathrm{for}}\hookrightarrow NL$ of the formal neighborhood $NL_{\mathrm{for}}$ of $L$ into $NL$. Then $\hat{\Phi}_t$ is a formal diffeomorphism between $NL_{\mathrm{for}}$ and the formal neighborhood $\hat{L}_M$ of $L$ inside $M$. Let 
$\Delta^{\mathrm{for}}_t=\hat{\Phi}_t^*(\Psi\bigr\vert_{\hat{L}_M})$. Then $\Delta^{\mathrm{for}}_0=\Delta^{\mathrm{for}}_{L,\tau_0}$ and $\Delta^{\mathrm{for}}_1=\Delta_{L,\tau_1}$. Moreover, writing $\hat{\Xi}_t$ for the formal diffeomorphism of $NL_{\mathrm{for}}$ relative to $L$ given by $\hat{\Xi}_t=\hat{\Phi}_0^{-1}\circ \hat{\Phi}_t$ we have
\[
\Delta^{\mathrm{for}}_t=\hat{\Phi}_t^*(\hat{\Phi}_0^{-1})^*\hat{\Phi}_0^*(\Psi\bigr\vert_{\hat{L}_M})=\hat{\Xi}_t^*\Delta^{\mathrm{for}}_0
\]
As $\hat{\Xi}_0=\mathrm{id}_{NL_{\mathrm{for}}}$, differentiating the above equation with respect to $t$ we find
\[
\frac{d}{dt}\Delta^{\mathrm{for}}_t=\mathcal{L}_{\hat{\xi}_t}\Delta^{\mathrm{for}}_t,
\]
where $\mathcal{L}_{\hat{\xi}_t}\Delta^{\mathrm{for}}_t$ is the Lie derivative of the tensor field $\Delta^{\mathrm{for}}_t$ with respect to the vector field $\hat{\xi}_t=\frac{d}{dt}\hat{\Xi}_t$. For every $t$, the vector field $\hat{\xi}_t$ is an element in $\Omega^0(NL_{\mathrm{for}},TNL_{\mathrm{for}})=(\g_L^{\mathrm{for}})_0$. Moreover, $\mathcal{L}_{\hat{\xi}_t}\Delta_t=[\hat{\xi}_t,\Delta^{\mathrm{for}}_t]^{FN}$. Thus, the family of elements $\Delta_t^{\mathrm{for}}$ satisfies
\[
\begin{cases}
\frac{d}{dt}\Delta_t^{\mathrm{for}}=[\hat{\xi}_t,\Delta_t^{\mathrm{for}}]^{FN}\\
\\
\Delta_0^{\mathrm{for}}=\Delta_{L,\tau_0}^{\mathrm{for}}\\
\\
\Delta_1^{\mathrm{for}}=\Delta_{L,\tau_1}^{\mathrm{for}}
\end{cases}
\]
and it is therefore a gauge equivalence between $\Delta_{\tau_0}^{L,\mathrm{for}}$ and $\Delta_{L,\tau_1}^{\mathrm{for}}$ in $\g_L^{\mathrm{for}}$. The final part of the statement follows from the following  

\begin{proposition}[Cattaneo \& Sch\"atz, cf.~{\cite[Theorem 3.2]{CS2008}}]\label{theor:CS}
Let $(\mathfrak g,\mathfrak a, \mathrm{j},  P, \Delta_0)$ and $(\mathfrak g, \mathfrak a, \mathrm{j},  P, \Delta_1)$ be $V$-data, and let $\mathfrak a[-1]_0$ and $\mathfrak a[-1]_1$ be the associated $L_\infty$-algebras. If $\Delta_0$ and $\Delta_1$ are gauge equivalent and they are intertwined by a gauge transformation preserving $\ker P$, then $\mathfrak a[-1]_0$ and $\mathfrak a[-1]_1$ are $L_\infty$-isomorphic.
\end{proposition}

\end{proof}
\begin{corollary}\label{corollary-isomorphism2}
Let $\tau_0$ and $\tau_1$ be two isotopic tubular neighborhoods of $L$ in $M$. Then the V-data $(\mathfrak g_{L},  \a_{L}, \mathrm{j}_{L}, P_{L},  \triangle_{L,\tau_0})$ and $(\mathfrak g_{L},  \a_{L}, P_{L}, \mathrm{j}_{L}, \triangle_{L,\tau_1})$ induce isomorphic $L_\infty$-algebra structures on $\Omega^*(L,NL)[-1]$.
\end{corollary}
\begin{proof}
Immediate from Proposition \ref{proposition-restriction} and Lemma \ref{lemma-gauge-equivalent}.
\end{proof}

Putting Proposition \ref{proposition-restriction} and Corollary \ref{corollary-isomorphism2} together, we obtain the following statement, which is a rephrasing of Theorem \ref{thm:linfty}.
\begin{proposition}
Let $\Psi \in \Om ^{*}(M, TM)$ be an odd square-zero  element, 
 and  $L$  a $ \Psi$-submanifold of $M$. Then the $\Z_2$-graded $L_\infty$-algebra structure on $\Om^*(L,NL)[-1]$ induced by the V-data $(\mathfrak g_{L},  \a_{L}, \mathrm{j}_{L}, P_{L},  \triangle_{L,\tau})$ is independent of the tubular neighborhood $\tau$, up to isomorphism.
\end{proposition}
\begin{proof} Given two tubular neighborhoods $\tau_0$ and $\tau_1$ of $L$ in $M$, there always exists a third tubular neighborhood  $\tilde{\tau}_1$ such that $\tau_0$ and $\tilde{\tau}_1$ are isotopic relative to $L$ and $\tilde{\tau}_1=\tau_1\circ\psi$ for a suitable diffeomorphism of $NL$ relative to $L$, see, e.g., \cite[Theorem 5.3]{H1997}. 
\end{proof}

\section{Deformations  of  $\Psi$-submanifolds}\label{sec:def}
Let $ \Psi \in \Om^{2l-1}(M, TM)$  be an odd degree, square zero element  and $L$   a closed $ \Psi$-submanifold in $M$. As in  the proof of Theorem \ref{thm:linfty}, we use a tubular neighborhood $\tau : NL \to U \subset M$ to identify the normal bundle $NL$ with an open neighborhood $U$ of $L$ in $M$, and we thus may replace $M$ by $NL$. In particular, we may regard $\Psi$ as a square zero element in  $\Om^*(NL, TNL)$.

A smooth small  deformation    of $L$ in $NL$  can be identified   with a (smooth) section $L \to NL$, i.e., with an element in $\Om^0(L, NL)$. In other words, when thinking of small deformations we implicitly identify $L$ with the image of the zero section $\mathbf 0 : L \to NL$. We say that
a section $s \colon L \to NL$  is {\it a $ \Psi$-section}, if  its image $s(L)$ is a $ \Psi$-submanifold in $NL$.
These have an elegant characterization in terms of the maps $\mathrm{j}_L: \Om^*(L, NL) \to \Om^* (NL, TNL)$ and $P_L: \Om^* (NL, TNL) \to \Om^* (L, NL)$ from Definition \ref{def-4.4}.

\begin{proposition}\label{prop:psisection} Let  $F_\Psi: \Gamma (NL) \to \Om^*(L, NL)$ be the map defined by
 \begin{equation}\label{eq:F_Psi}
 F_\Psi  (s) : = P_L ( \exp  \mathrm{j}_L(-s)^*  \Psi ).
 \end{equation}
 Then a  section $s: L \to NL$  is a  $ \Psi$-section
if and only if  $F_\Psi (s)  = 0  \in \Om^*(L, NL)$.
\end{proposition}

\begin{proof}
Let $x \in L$. We begin with two simple remarks. First of all, for $v \in T_x L$ we have
\begin{equation}\label{eq:1_LUCA}
\exp \mathrm{j}_L(s)_\ast v = T_x s \cdot v.
\end{equation}
Second, let $w \in T_{s(x)} NL$. Then $w$ can be uniquely written as $w = w_s + w_N$ where $w_s$ is tangent to $s(L)$ and $w_N$ is a tangent vector vertical with respect to projection $NL \to L$. In particular $w_N$ is the vertical lift of a, necessarily unique, vector in $N_x L$ that we denote $w_N^{\downarrow}$. Finally, we have
\begin{equation}\label{eq:2_LUCA}
\pi_{NL} \exp \mathrm{j}_L(-s)_\ast w = w_N^\downarrow,
\end{equation}
where  ${\pi^{}_{NL}}: TNL \to NL$ is the projection to the base.
Both (\ref{eq:1_LUCA}) and (\ref{eq:2_LUCA}) can be easily checked, e.g.~in local coordinates.
Now, we compute $F_\Psi (s)$ explicitly. So, let $v_1, \ldots, v_{2l-1} \in T_x L$. Then
\[
\begin{aligned}
& F_{\Psi}(s)_x (v_1, \ldots, v_{2l-1}) \\
& = P_L(\exp \mathrm{j}_L(-s)^\ast \Psi)_x (v_1, \ldots, v_{2l-1}) \\
& = \pi_{NL}^{}  \exp \mathrm{j}_L(-s)_\ast \left( \Psi_{s(x)} (\exp \mathrm{j}_L(s)_\ast v_1, \ldots, \exp \mathrm{j}_L(s)_\ast v_{2l-1})\right) \\
& = \pi_{NL}^{} \exp \mathrm{j}_L(-s)_\ast \left(\Psi_{s(x)} (T_x s \cdot v_1, \ldots, T_x s \cdot v_{2l-1}) \right) &\mathrm {by}\; (\ref{eq:1_LUCA}) \\
& = \Psi_{s(x)} (T_x s \cdot v_1, \ldots, T_x s \cdot v_{2l-1})_N^\downarrow. &\mathrm {by}\; (\ref{eq:2_LUCA})
 \end{aligned}
\] 
This shows that $F_\Psi (s)= 0$ if and only if $ \Psi (w_1, \ldots, w_{2l-1})$ is tangent to $s(L)$ for all $w_1, \ldots, w_{2l-1}$ tangent to $s(L)$, i.e.~$s(L)$ is a $\Psi$-submanifold.
\end{proof}

As a step towards our investigation of smooth deformations of a $\Psi$-submanifold $L$, in this  section we first study formal  $\Psi$-deformations of  $L$
 (Definition \ref{def:formalpsi})   and show how they are governed  by  the   Maurer-Cartan equation   of the  $L_\infty$-algebra attached to $L$  (Proposition \ref{prop:mc}). Then we   study  infinitesimal and smooth deformations   of $\Psi$-submanifolds    and  prove the main theorem  of this section   on the local structure of the  pre-moduli space  of analytic $\Psi$-submanifolds in  an analytic  manifold  $M$  (Theorem  \ref{thm:Main-Psi}) under the condition that $\Psi$ is multi-symplectic  (Definition  \ref{def:non-degenerate}). 
 
 \subsection{Formal deformations of $\Psi$-submanifolds}\label{subs:formal}
Let $\eps$  be a formal parameter. Let us recall that a formal series $s(\eps) = \sum_{i=0}^\infty\eps ^i s_i  \in \Gamma (NL)[[\eps]], s_i \in \Gamma (NS)$  such that
$s_0 =0$ is called  {\it  a formal deformation of $L$}, and $s_1 \in \Gamma(NL)$ is called its {\em initial velocity}.

Denote by  $\X(NL)$ and $\T^{(r,s)}(NL)$ the space  of smooth vector  fields and $(r,s)$-tensor fields on  $NL$, where $\X(NL)$ is interpreted as the derivations of the (commutative) algebra of smooth functions $C^\infty(NL)$. The Lie derivative of tensor fields naturally extends to formal power series; for a {\em formal vector field } $\xi(\eps) := \sum_0^\infty \eps^i \xi_i \in \X(NL)[[\eps]]$ and a {\em formal $(r,s)$-tensor field }$T(\eps) := \sum_0^\infty \eps^i T_i \in \T^{(r,s)}(NL)[[\eps]]$ we define the {\em formal Lie derivative}
\begin{equation} \label{eq:LieDer}
\Ll_{\xi (\eps)}  T(\eps) : = \sum_{ i=0} ^\infty \eps ^k \sum_{ i+j  = k} \Ll_{\xi_i } T_j
\end{equation}
and the formal exponential acting on $\T^{(r,s)}(NL)[[\eps]]$ as
\begin{equation}\label{eq:bch}
 \exp \Ll_{\xi (\eps)} : = \sum_{n =0} ^\infty \frac{1}{n !} \Ll^n _{ \xi (\eps)}.
 \end{equation}

Any  section $s\colon L \to NL$ defines the constant vector field $\mathrm{j}_L(s)$ on $NL$: the flow on $NL$ generated by the vector field $\mathrm{j}_L(s)$ on $NL$ is given by $\Phi^t_{\mathrm{j}_L(s)} y_x = y_x + t s(x)$ for all $y_x \in N_xL$. The same applies to formal series $s(\eps) = \sum_{i=0}^\infty\eps ^i s_i  \in \Gamma (NL)[[\eps]]$.

Proposition (\ref{prop:psisection}) motivates  the following
\begin{definition}\label{def:formalpsi}  A formal   deformation $s(\eps)$ of $L$ is called {\it a $\Psi$-formal deformation},  if
$F_\Psi (s(\eps)) : = P_L (\exp  \Ll_{ \mathrm{j}_L (-s(\eps))}   \Psi) = 0 \in \Om^*(L, NL)[[ \eps]]$.

An {\it infinitesimal  $\Psi$-deformation  of $L$} is a section $s: L \to NL$ for which $F_\Psi(\eps s(L)) = O(\eps^2)$.
\end{definition}

If $s(\eps) = \sum_{i =0}^\infty \eps^i s_i$ is a formal    $ \Psi$-deformation of $L$, then its initial velocity $s_1$ evidently is an  infinitesimal  $\Psi$-deformation. Conversely, given an  infinitesimal     $ \Psi$-deformation  $s_1$,  we say that $s_1$  is {\it  unobstructed}, if  there exists  a  formal
$ \Psi$-deformation with initial velocity $s_1$. If all infinitesimal deformations are unobstructed, then we say that \emph{the formal deformation problem is unobstructed}. Otherwise it 
is  {\it obstructed} (cf.  \cite[\S  10]{LO2016}, \cite[Remark 4.8]{LOTV2014}, \cite[Definition 4.8]{LS2014}).

Recall  the multibracket $\l_n$    of the  $L_\infty$-algebra
associated  to   a $\Psi$-submanifold (Theorem \ref{thm:linfty})  has been  defined  in (\ref{eq:multib})  in Proposition \ref{prop:voronov}. 

\begin{proposition}\label{prop:mc} The  formal  $\Psi$-deformations  of  $L$ are governed  by  the  $L_\infty$-algebra $\Om^*(L, NL)$. Namely, a formal  deformation $s(\eps)$ of $L$ is a   formal $\Psi$-deformation if and only if  $-s(\eps)$ is a solution 
	of the (formal)  Maurer-Cartan  equation
	\begin{equation}\label{eq:mc}
	MC (s (\eps)) : = \sum_{n =1}^\infty \frac{1}{n!}\l_n (s(\eps), \cdots, s(\eps)) = 0.
	\end{equation}
	\begin{proof}  
	From  the definition of $\l_n$ we get
		\begin{equation}\label{eq:PLie}
		P_L(\Ll_{\mathrm{j}_L(-s)}  ^n   \Psi)  = \l_n (s, \cdots,  s), \quad n \ge 1,
		\end{equation}
		for $s \in \Gamma (NL)$. This implies the identity of formal power series 
\[
		P_L(\Ll_{\mathrm{j}_L(-s(\varepsilon))}  ^n   \Psi)  = \l_n (s(\varepsilon), \cdots,  s(\varepsilon)), \quad n \ge 1,
\]	
for any $s(\eps) = \sum_{i=0}^\infty\eps ^i s_i  \in \Gamma (NL)[[\eps]]$ 
and so
\[
MC (s (\eps)) = \sum_{n =1}^\infty  \frac{1}{n!}P_L(\Ll_{\mathrm{j}_L(-s(\varepsilon))}  ^n   \Psi) = P_L(\exp \Ll_{\mathrm{j}_L(-s(\varepsilon))}     \Psi)
\]
for any formal series $s(\eps) = \sum_{i=0}^\infty\eps ^i s_i  \in \Gamma (NL)[[\eps]]$.
	\end{proof}
\end{proposition}

\begin{corollary}\label{cor:infpsi} Let $s : L \to NL$ be a smooth section. Then $\varepsilon s$ 
is an infinitesimal $\Psi$-deformation of  $L$ if and only if $\l_1(s) =0$, i.e.,  if and only if $s \in \ker d_{\mathbf{0}} F_\Psi$, where $d_{\mathbf{0}} F_\Psi$ is the differential of $F_\Psi$ at the point $\mathbf{0}$ of $\Gamma(NL)$.
\end{corollary}
We shall denote the space of infinitesimal $\Psi$-deformations as
\begin{equation} \label{eq:JPsi}
J_\Psi(L) := \ker \l_1 = \ker d_{\mathbf{0}} F_\Psi.
\end{equation}

\subsection{Smooth and infinitesimal deformations of $\Psi$-submanifolds}\label{subs:smoothpsi}

\begin{definition}\label{def:psidef} A {\it smooth  $ \Psi$-deformation  of $L$} is a smooth one-parameter deformation $\{ s_t\}$  of the zero section of the vector bundle   $NL \to L$ such that each   section in the family is  a {\it $\Psi$-section}.
\end{definition}

Clearly  if $\{ s_t\}$  is a smooth $\Psi$-deformation, then    the section $\frac{ds_t}{dt}_{|t = 0}: L \to NL$ is an infinitesimal $ \Psi$-deformation. 
More generally, the Taylor expansion
\[
\sum_{n=0}^\infty \left(\left. \frac{d^n }{dt^n }\right\vert_{t=0} s_t\right)\varepsilon^n
\]
of a smooth $\Psi$-deformation $\{ s_t\}$ is a formal $\Psi$-deformations. Smooth obstructedness/unobstructedness are defined in a similar way  as  formal  obstructedness/unobstructedness.
 
For $\Psi \in \Om(M, TM)$    denote by $\Di_{\Psi} (M)$ the   subgroup  of the     diffeomorphism group $\Di (M)$ whose elements preserve $\Psi$.

\begin{definition}\label{def:moduli} Given $\Psi \in \Om^\ast(M, TM)$ and a homology class $\alpha \in H_l(M, {\mathbb Z})$, we  denote  by $\Mm_\Psi (\alpha)$ the   set of
all  closed $l$-dimensional $\Psi$-submanifolds representing the homology class $\alpha$ and call it {\it the pre-moduli space of $\alpha$}. The quotient  $\Mm_\Psi (\alpha) /  \Di_\Psi  (M)$  is called {\it the  moduli space  of $\Psi$-submanifolds of homology class $\alpha$}. 

Furthermore, for a $\Psi$-submanifold $L \subset M$, we denote by $\Mm_\Psi(L)$ the pre-moduli space of closed $\Psi$-submanifolds in  $M$ that are obtained   from $L$ by smooth $\Psi$-deformations, so that $\Mm_\Psi(L) \subset \Mm_\Psi([L])$.
\end{definition}

Here we will work with the pre-moduli spaces only and will not discuss the \emph{moduli problem}. But note that in most applications, $\Di_\Psi  (M)$ is a (finite dimensional) Lie group.
Under suitable analiticity and nondegeneracy conditions on $\Psi$ this will imply that the moduli
		space  of $\Psi$-submanifolds  in the connected component  of  $L$ is a finite dimensional analytic space, see Theorem \ref{thm:Main-Psi} and Remark \ref{rem:simul}.

Since locally, $\Mm_\Psi (L)$ is the set of $C^1$-small solutions of the equation $F_\Psi(s) = 0$, we shall use the tools provided in the proof of \cite[Theorem 4.9]{LS2014}, see also the pioneering paper by  Koiso \cite{Koiso1983} for a similar idea.

Being a differential operator, $F_\Psi$ extends for each $k \geq 1$
to a map denoted by the same symbol
\begin{equation} \label{eq:FPsik}
F_\Psi:  L^2_k\Om^0(L, NL) \longrightarrow  L^2_{k-1}\Om^{l-1}(L, NL)
\end{equation}
where $L^2_k\Om^l(L, NL)$ denotes the Sobolev space of $L^2_k$ $l$-forms on $L$ with values in $NL$, i.e., the completion of $\Om^l(L, NL)$ in the $L^2_k$ norm.

\begin{proposition} \label{prop:anal2k}
Let $M$ be a real analytic manifold, $\Psi \in \Om^{l-1}(M, TM)$ analytic and $L \subset M$ a closed analytic $\Psi$-submanifold. Then for each $k$, the map $F_\Psi$ from (\ref{eq:FPsik}) is analytic in a neighborhood  of the zero-form  $\mathbf{0}\in L^2_k\Om^0(L, NL)$,
\begin{equation} \label{eq:deriv-F(st)}
F_\Psi(s) = P_L (\exp \Ll_{\mathrm{j}_L(-s)} \Psi) = \sum_{n=1}^\infty \dfrac1{n!} P_L (\Ll_{\mathrm{j}_L(-s)}^n  \Psi) \stackrel{(\ref{eq:PLie})} = \sum_{n=1}^\infty \dfrac1{n!} \l_n(s, \cdots, s).
\end{equation}
\end{proposition}

\begin{proof}
We follow an approach similar to \cite{LS2014}. First  we consider   the restriction  of  the map $F_\Psi$ defined  in (\ref{eq:FPsik})  to the space  $C^k\Om^0(L,NL) \subset L_k^2 \Om^0(L,NL)$. Abusing notation,  the restriction  is also denoted by $F_\Psi$. Note that  the image  $F_\Psi (C^k \Om^0 (L, NL))$ belongs  to $C^{k-1}\Om^l(L, NL)$. Now we  consider  spaces $C^k\Om^0 (L, NL)$ and $C^{k-1}\Om^l(L, NL)$ as Banach  spaces  with $C^k$-norm  and $C^{k-1}$-norm respectively.

Choose a real analytic local trivialization $(x^i, y^r)$ of $NL$ where $(x^i)$ are the coordinates on $L$. For any $C^k$-function $f(x^i)$ in this neighbourhood we define its {\em $C^k$-norm at $x$} as
\[
\|f\|_{C^k; x} : = \sum_{ | I|  \le  k } \|(D_I s)_x \|,
\]
and likewise, for a section $s \in C^k(L, NL)$ we define its $C^k$-norm (at $x$) as
\begin{equation}\label{eq:cklocal}
\|s\|_{C^k; x} : = \sum_{ | I|  \le  k } \|(D_I s)_x \|,  \text{ and }   \| s\|_{ C^k} :  =  \sup _{ p \in L}  \|s\| _{ C^k; p}
\end{equation}
where the  sum is  taken of  all multi-indices  $I$, and $D_I$ denotes multiple partial derivatives  with respect to  the given coordinates. Observe that there is a constant $C_k$ such that for all $C^k$-functions $f$ and $g$
\begin{equation} \label{eq:D_I(fg)}
\|fg\|_{C^k; x} \leq C_k \|f\|_{C^k; x}\; \|g\|_{C^k; x},
\end{equation}

In these coordinates, $\Psi \in \Om^{2k-1}(NL, TNL)$
takes the form
 \begin{equation}\label{eq:Psi1}
\Psi = \sum_{|I|+|R| = 2k-1}  dx ^{I} \wedge dy  ^{R} \otimes \left(f^r_{I; R}(x,y)   \frac{\p}{\p y^r} + f^i_{I; R}(x,y) \frac{\p}{\p x^i}\right)
 \end{equation}
 where $I, R$ are  skew-symmetric multi-indices, and where we use the Einstein convention of summation over repeated indices. A section $s \in \Gamma(L, NL)$ and its graph $s(L) \subset NL$ are given as
 \begin{equation} \label{eq:def-s(x)}
s(x) = s^r(x) \frac{\p}{\p y^r} \qquad \text{and} \qquad y^r = s^r(x),
 \end{equation}
respectively, and now a straightforward calculation yields
 \begin{equation}\label{eq:fk1}
\begin{aligned}
& F_\Psi(t s)\\
& = \sum_{|I|+|R| = 2k-1} (-t)^{|R|} \left(f^r_{I; R}(x,-t s(x))  + t\; f^i_{I; R}(x, -t s(x)) \frac{\p s^r}{\p x^i} (x)\right)
 dx^I \wedge ds^{R} \otimes   \frac{\p}{\p y^r}
 \end{aligned}
 \end{equation}
 where, for $R = (r^1, \cdots, r^p)$ we set
 \begin{equation}\label{eq:s}
 ds^R = ds^{r_1} \wedge \cdots \wedge ds^{r_p} = \frac{\p s^{r_1}}{\p x^{i_1}} \cdots \frac{\p s^{r_p}}{\p x^{i_p}} 
 dx^{i_1} \wedge \cdots \wedge dx^{i_p}.
 \end{equation}
In particular,
\[
\left. \dfrac{d^n}{dt^n}\right|_{t=0} F_\Psi(t s) = \sum_{|I| = l-1} p_{I}^r\left(x; s^r, \frac{\p s^r}{\p x^i} \right) dx^I \otimes \frac{\p}{\p y^r} ,
\]
where each $p_{I}^r$ is a homogeneous polynomial of degree $n$ in the variables $\left(s^r, \frac{\p s^r}{\p x^i}\right)$ whose coefficients are linear combinations of functions of the form
\[
D_S f_{I;R}^r, D_S f_{I;R}^i, \qquad |S| \leq n,
\]
where $S = (r^1, \cdots r^n)$ is a multi-index in the $y^r$-variables only. Since $f_{I;R}^r, f_{I;R}^i$ are real analytic, it follows (cf. \cite[Proposition 2.2.10]{KP2002}) that -- after possibly shrinking the coordinate neighborhood -- there are positive constants $A, K$ such that for any multi-index $I = (i_1, \cdots i_l)$ and for all $x$, $|D_I D_S f_{I;R}^r|, |D_I D_S f_{I;R}^i| \leq n! l! A K^n K^l$ and hence,
\begin{equation} \label{eq:D_I D_S f}
\|D_S f_{I;R}^r\|_{C^k;x}, \|D_S f_{I;R}^i\|_{C^k; x} \leq n! \tilde A K^n
\end{equation}
for all $x$ and $|S| \leq n$, with fixed $\tilde A, K > 0$.
Finally,
\begin{equation} \label{eq:D_I s^r}
\|s^r\|_{C^k;x} \leq C_0 \|s\|_{C^{k}; x} \leq C_0 \|s\|_{C^{k+1}; x}, \quad  \left\| \frac{\p s^r}{\p x^i}(x)\right\| \leq C_0 \|s\|_{C^{k+1}; x}
\end{equation}
for some constant $C_0$. Therefore, as $p_{I}^r$, is a homogeneous polynomial, it follows from (\ref{eq:D_I D_S f}), (\ref{eq:D_I(fg)}) and (\ref{eq:D_I s^r}) that for all $x$,
\[
\left\|p_{I}^r\left(x; s^r(x), \frac{\p s^r}{\p x^i}(x) \right)\right\|_{C^k;x} \leq n! \tilde A (C_0 C_k K)^n \|s\|^n_{C^{k+1}; x},
\]
whence in this coordinate neighborhood
\begin{equation} \label{eq:estim-m-deriv-loc}
\left\|\left. \dfrac{d^n}{dt^n}\right|_{t=0} F_\Psi(t s)\right\|_{C^k;x} \leq n! A_0 K_0^n \|s\|^n_{C^{k+1}; x}
\end{equation}
for all $x$, and since by compactness $L$ may be covered by finitely many such neighborhoods, we may assume that (\ref{eq:estim-m-deriv-loc}) holds for all $x \in L$ for fixed constants $A_0, K_0 > 0$. That is, 
\begin{equation} \label{eq:estim-m-deriv}
\left\|\left. \dfrac{d^n}{dt^n}\right|_{t=0} F_\Psi(t s)\right\|_{C^k} \leq n! A_0 K_0^n \|s\|^n_{C^{k+1}}.
\end{equation}

By \cite[Lemma 6.2]{LS2014}, the estimate (\ref{eq:estim-m-deriv})  implies  that the  map $F_\Psi: C^k \Om^0 (L, NL)  \to  C^{k-1}\Om^{l-1} (L, NL)$ is an analytic map between    Banach spaces.  Since $L$ is compact,
as in \cite{LS2014},  this implies  that   the  map $F_\Psi: L^2_k \Om^0(L, NL) \to L^2_{k-1} \Om^{l-1}(L, NL)$ is also an  analytic map between    Banach  spaces.
\end{proof}

In order to utilize this analyticity, we shall need some regularity on the linearization $d_{\mathbf{0}}F_\Psi$. For $\xi \in T^\ast_xM$, we define the linear map
 \begin{equation} \label{eq:sigmaxi}
 \sigma_\xi: \Lambda^{l-1} T_x^\ast M \otimes T_xM \longrightarrow \Lambda^l T^\ast_xM, \quad \alpha^{l-1} \otimes v \longmapsto \xi(v) \xi \wedge \alpha.
 \end{equation}
 \begin{definition} \label{def:non-degenerate}
 We call $\Psi \in \Om^{l-1}(M, TM)$ {\em multi-symplectic} if $\sigma_\xi \Psi \neq 0$ for all $\xi \neq 0$. We say that $\Psi$ is {\em multi-symplectic on $L$} for a $\Psi$-submanifold $L$, if $\Psi_{|L}$ is multi-symplectic in $\Om^{l-1}(L, TL)$.
\end{definition}

This terminology generalizes the notion of multi-symplecticity of differential forms, as it follows from (\ref{eq:def-partial}) that $\Psi = \hat \varphi$ is multi-symplectic (on $L$) iff $\varphi$ ($\varphi_{|L}$, respectively) is multi-symplectic, meaning that $\imath_\xi \varphi = 0$ only if $\xi = 0$.

\begin{proposition} \label{prop:non-degenerate}
Let $\Psi$ be multi-symplectic on the $\Psi$-submanifold $L \subset M$. Then $d_{\bf 0} F_\Psi$ is an overdetermined elliptic differential operator. In particular, this is the case if $\Psi = \hat \varphi$ and $\varphi_{|L}$ is multi-symplectic.
\end{proposition}

\begin{proof}
We pick coordinates $(x^i, y^r)$ as in the proof of Proposition \ref{prop:anal2k}. Then (\ref{eq:fk1}) yields
\begin{align*}
d_0F_\Psi(s) &= \l_1(s) = \left. \dfrac d{dt}\right|_{t=0} F_\Psi(ts)\\
& = \sum_{|I| = 2k-1} \left( f^i_{I; \emptyset}(x, 0) \dfrac{\p s^r}{\p x^i}(x)- s^u \dfrac{\p}{\p y^u} f^r_{I, \emptyset}(x,0)\right) dx^I \otimes \dfrac \p{\p y^r}\\
& \qquad - \sum_{|J| = 2k-2} f^r_{J; u} (x, 0) \dfrac{\p s^u}{\p x^i} dx^J \wedge dx^i \otimes \dfrac \p{\p y^r}.
\end{align*}
Thus, for $\xi = \xi_i dx^i \in T^\ast_{x_0}L$, the symbol of $d_0F_\Psi$ is
\begin{align*}
& \sigma_\xi d_0F_\Psi (s) =\\
&= \sum_{|I| = 2k-1} f^i_{I; \emptyset}(x_0, 0) \xi_i s^r dx^I \otimes \dfrac \p{\p y^r}
-\sum_{|J| = 2k-2} f^r_{J; u} (x_0, 0) \xi_i s^u dx^J \wedge dx^i \otimes \dfrac \p{\p y^r}\\
 &\stackrel{(\ref{eq:def-s(x)})} = \sum_{|I| = 2k-1} f^i_{I; \emptyset}(x_0, 0) \xi_i dx^I \otimes s
  - \xi \wedge \sum_{|J| = 2k-2} f^r_{J; u} (x_0, 0) s^u dx^J \otimes \dfrac \p{\p y^r}
\end{align*}
which implies
\[
\xi \wedge \sigma_\xi d_0F_\Psi (s) =   \xi \wedge \sum_{|I| = 2k-1} f^i_{I; \emptyset}(x_0, 0) \xi_i dx^I\otimes s \stackrel{(\ref{eq:sigmaxi}), (\ref{eq:Psi1})}=  \iota_\xi \Psi_{|L} \otimes s.
\]
Since $\Psi_{|L}$ is multi-symplectic, $\iota_\xi \Psi_{|L} \neq 0$ for all $\xi \neq 0$, whence the symbol $\sigma_\xi d_0F_\Psi$ is injective for all $\xi \neq 0$, showing the assertion.
\end{proof}

We are almost in the position now to construct a local analytic chart on $\Mm_\Psi (L)$  using the \emph{Inverse Function Theorem} (IFT)  for  analytic mappings between   real analytic  Banach  manifolds \cite{Douady1966},  see   \cite[Appendix]{LS2014} for a short account. However, the main difference to the situation handled there is that we do not know a priori if the space $J_\Psi(L)$ of infinitesimal $\Psi$-deformations is finite dimensional, whence we need to impose this as an additional condition.

\begin{theorem} \label{thm:Main-Psi}
Let $M$ be an analytic manifold with an analytic section $\Psi \in \Om^{l-1}(M, TM)$. 
\begin{enumerate}
\item If $L \subset M$ is an analytic $\Psi$-submanifold such that $\Psi$ is multi-symplectic on $L$ and $J_\Psi(L)$ is finite dimensional, then the  pre-moduli space $\Mm_\Psi (L)$ of all $\Psi$-submanifolds  $C^1$-close to $L$ forms a finite dimensional analytic variety.

\item If any $\Psi$-submanifold in $\Mm_\Psi (L)$ shares the properties given in (1), then $\Mm_\Psi (L)$ is a finite dimensional analytic space.

\end{enumerate}
\end{theorem}

\begin{proof}
 As $\Psi$ is fixed, we shall simply write $F$ instead of $F_\Psi$. Since $\Psi$ is multi-symplectic on $L$ and hence $d_{\mathbf 0}F$ is overdetermined  elliptic by Proposition \ref{prop:non-degenerate}, we  have the following $L^2$-orthogonal decomposition (see e.g. \cite[Corollary 32, p. 464]{Besse1987})

\begin{equation}\label{eq:decom1}
L^2\Om^{l-1}(L, NL) = d_{\mathbf 0}F (L^2_1 \Gamma (NL)) \oplus  \bigl(\ker  (d_{\mathbf 0}F) ^* \cap L^2\Om^{l-1}(L, NL)\bigr).
\end{equation}

\begin{itemize}
\item Let  $\Pi_1:  L^2  \Om^{l-1} (L, NL)  \to d_{\mathbf 0}F (L^2_1 \Gamma (NL))$ be the orthogonal projection  with respect to     the  decomposition in (\ref{eq:decom1}). 
 Being bounded linear, 
$\Pi_1$ is  an  analytic map between  Banach spaces.
\item Let  $ U(\mathbf 0)$  denote   an open  neighborhood  of $\mathbf 0$ in $L^2_1(\Gamma (NL))$  such that   the restriction of
the map $F$  to $U(\mathbf 0)$ is analytic. The existence  of $U(\mathbf 0)$ is ensured by Proposition \ref{prop:anal2k}.
\item Denote by $\pi: L^2_1 \Gamma (NL) \to J_\Psi(L)$  the orthogonal   projection. Since   $J_\Psi(L)$ is assumed to be finite dimensional, $\pi$ is bounded linear and hence analytic.
\end{itemize}
 Then we set
\begin{equation}\label{eq:proj}
\hat F: = \pi \oplus  \left( - \Pi _1 \circ F \right):    L ^2_1 (\Gamma(NL) )\supset U(\mathbf 0) \to J_\Psi(L)  \oplus  d_{\mathbf 0}F (L^2_1 \Gamma (NL)).
\end{equation}

By Proposition \ref{prop:anal2k}, the map $\hat F$   is analytic in $U(\mathbf 0)$   and its differential  at $\mathbf 0$  is an isomorphism. Therefore the IFT  for analytic mappings of Banach spaces   implies that  there  is an  analytic inverse of $\hat  F$
$$ G:  V(0, \mathbf 0) \to U  (\mathbf 0)$$
where $V(0, \mathbf 0)$   is an open neighborhood of $(0, \mathbf 0) \in  J_\Psi(L) \oplus  d_{\mathbf 0}F(L^2_1 \Gamma(NL))$.

Let $V^{J_\Psi} ( 0, \mathbf 0): = V (0, \mathbf 0)  \cap (J_\Psi(L), \mathbf 0)$. 
Next we define  the map
\begin{equation}\label{eq:inverse}
\tau:  V^{J_\Psi} (0, \mathbf 0) \to J_\Psi(L), \:     s \mapsto  \pi \circ G (s) - i(s)
\end{equation}
where  $i: (J_\Psi(L), \mathbf 0)  \to J_\Psi(L)$  is the natural   identification map. We now assert that

\begin{enumerate}
\item The map $\tau$ is analytic.
\item The  restriction  of the    projection $\pi$  to   $F^{-1} (\mathbf 0) \cap  U( \mathbf 0)$ is injective.
\item An element  $y \in V^{J_\Psi}( 0, \mathbf 0)$  belongs to $\tau^{-1}(0)$ if and only if  $y = i ^{-1}\circ \pi(z)$ for some   $z\in (\Pi_1 \circ F)^{-1}(\mathbf 0)$. 
\end{enumerate}

The first statement holds as both $\pi$ and $G$  are analytic maps, whereas the second holds since $\hat F$ is locally invertible at the origin.

To see the last assertion, let us first proof the ``if"-part. Assume that  $y = i^{-1} \circ \pi (z)$   and $\Pi_1\circ F(z) = \mathbf 0$.  Then $\hat F (z) = i^{-1}\circ  \pi (z) = y$.
It follows $ z= G (y)$ and $\tau (y)  = \pi \circ  G (y) - i(y) = \pi (z)  -  \pi (z) = 0$, which  proves the  ``if"-assertion.   

Now assume that $\tau (y) = 0$.  Then  $\pi \circ  G (y)   = i (y)$.  Set $z = G (y) $. Then  $\hat F (z)  = y = \pi (z)$ and therefore, $\Pi_1 \circ F(z) =\mathbf 0$, so that the third assertion is shown as well.

We now consider the restriction $F: \tau^{-1}( 0) \to \ker \Pi_1$. Then $F^{-1}({\bf 0})$ consists of the common zeroes of the analytic functions $F_\xi(\cdot) := \langle \xi, F(\cdot)\rangle$ on $\tau^{-1}(0)$ for all $\xi \in \ker \Pi_1$ and where $\langle\cdot,\cdot\rangle$ refers to the scalar product on $L^2\Om^{l-1}(L, NL)$.

Since $\tau^{-1}(0) \subset V^{J_\Psi} (0, \mathbf 0)$ 
is a finite dimensional real analytic variety, the ring of germs of analytic functions at $\mathbf 0$ is Noetherian \cite[Theorem I.9]{Frisch}. Therefore, $F^{-1}({\bf 0})$ is given as the zero set of {\em finitely many } analytic functions $F_{\xi_1}, \ldots, F_{\xi_N}$. In other words, there is an analytic function
\[
\hat \tau:= (\tau, F_{\xi_1}, \ldots, F_{\xi_N}): V^J(0, {\bf 0}) \longrightarrow J_\Psi(L) \oplus \R^N
\]
for some finite number $N$ such that $F^{-1}(y) = 0$ iff $y = i^{-1} \circ \pi(z)$ for some $z \in F^{-1}(\mathbf 0)$. This allows to  identify     the $C^1$-neighborhood  $F^{-1}(\mathbf 0)  \cap U (\mathbf 0)$  of $L$ in the   pre-moduli  space  $\Mm_\Psi (L)$
with  the  pre-image  $\hat \tau^{-1} (0) $   in the neighborhood of $  0 \in J_\Psi(L)$  via the map  $ i ^{-1} \circ \pi: F^{-1}(\mathbf 0)  \cap U (\mathbf 0)  \to \hat \tau ^{-1} ( 0)$, 
where $\hat \tau $  is an analytic map between  open neighborhoods of     finite dimensional   vector spaces. 
Since    $F^{-1}(\mathbf 0) \cap  U (\mathbf  0)$   
models   a $C^1$-neighborhood  $U(L)$ of $L$ in $\Mm_{\Psi} (L)$,
this      completes  the proof of the first statement of the theorem.

 To show the second statement, assume that $L_1 \in \Mm_\Psi(L)$  lies  in  a $C^1$-neighborhood of   $L$, so that $L_1 =  s(L)$ for some analytic section
$s$  of  $NL$.  Then   $s$  induces, via $\exp \mathrm{j}_L(s)$, an invertible analytic  map  between  the Sobolev spaces $L^2_1  \Gamma(NL)$ and $ L^2_1 \Gamma (NL_1)$
as well as   an invertible analytic   map between  $L^2_1 \Om^{l-1}(L, NL)$ and  $L^2_1 \Om^{l-1} (L, NL)$.

This means that the charts on $U_1(L_1)$  constructed  via
 maps $\pi ^{L_1}$, $\tau ^{L_1}$ 
are equivalent  to   the analytic structure  induced from  the one on $U(L)$. In other words, any two analytic charts are compatible, which completes the proof. 
\end{proof}

\section{Deformations of $\varphi$-calibrated  submanifolds}\label{subs:calibr}

In this    section we consider   smooth $\hat \varphi$-deformation  of a  closed $\varphi$-calibrated submanifold  $L$, where    $\varphi^l$ is
a parallel  calibration  on the  Riemannian  manifold $(M,g)$.  First,  we prove  that the  space  of all infinitesimal  $\hat \varphi$-deformations  of $L$ coincides    with the space  of Jacobi  vector fields  on $L$, regarding  $L$ as  a minimal submanifold (Proposition \ref{prop:infinitesimal}).  Then we prove  our main theorem  stating that the  formal and smooth  deformations of  a closed  $\varphi$-calibrated  submanifold   are encoded in its  cananically associated $\Z_2$-graded  strongly homotopy Lie algebra (Theorem \ref{thm:main}). In Remark \ref{rem:simul}  we discuss  some related  results. Then we   revisit  the deformation theory of complex submanifolds using the methods developed in the present paper (Theorem \ref{thm:c}). Finally, in the last subsection  we summarize the achievements of   the present  paper.

\subsection{Infinitesimal deformations of $\hat \varphi$-submanifolds}

Given a {\it parallel}  calibration $\varphi$   and  a  closed  $\varphi$-calibrated submanifold  $L$, the premoduli space
$\Mm_{\hat\varphi}(L)$  consists  of all  closed minimal  submanifolds  that are  obtained  from $L$ by smooth deformation.  The following  Proposition, in which we do not require the calibration $\varphi$  to be  parallel,  is an  infinitesimal  analogue
of Corollary \ref{cor:caldeform}.

\begin{proposition}\label{prop:infinitesimal} 
Let $\varphi \in \Om^l(M)$ be a calibration and $L \subset M$ closed such that either
\begin{enumerate}
\item $L$ is $\varphi$-calibrated, or
\item $L$ is a $\hat \varphi$-submanifold such that $\varphi_{|L} \not = 0$, and $\varphi$ is parallel.
\end{enumerate}
Then $L$ is minimal, and
\begin{equation} \label{eq:compareJ}
J_{\hat \varphi} (L) \subset J(L),
\end{equation}
where $J_{\hat \varphi} (L)$ is given in (\ref{eq:JPsi}) and $J(L)$ is the space of Jacobi 
fields    on $L$. Moreover, under condition (1), equality holds in (\ref{eq:compareJ}).
\end{proposition}

\begin{proof} The minimality of $L$ is evident if $L$ is $\varphi$-calibrated, and it follows from Theorem \ref{thm:crit} under condition (2). 
Assume that  $L$ is  a   closed $\hat \varphi$-submanifold  and
 $s\in  \Gamma (NL)$ is an infinitesimal   $\hat \varphi$-deformation. Let  us   recall that  $ \psi_t = \exp \bigl(\mathrm{j}_L(ts)\bigr)$. Since $L$ is compact, there exist  a  positive number $A$  and a positive number  $\eps_0$ such that,  for any $ x\in L$
 \begin{equation}\label{eq:jacobi}
 |\pr \bigl((\psi_t)^* \hat \varphi\bigr)_{|L}(x) |  \le A \cdot  t^2
 \end{equation}
 for  any $t \le \eps_0$.
 Denote by $G_\varphi(x)$  the  space of  unit  decomposable  $l$-vectors $w$ in $G_l (T_x M)$ such that
 $\varphi (w) =1$. 
 Denote by $\rho$ the  distance on  the Grassmannian  $G_l( T_xM)$   induced  by the Riemannian metric on $T_xM$.
 
 \ 
 
We shall abbreviate $L_t : = \psi_t (L)$, $x_t : = \psi _t (x)$  and $\hat \varphi_t : = (\psi_t)^* \hat \varphi$.
 
 \begin{lemma}\label{lem:jacobi} The inequality (\ref{eq:jacobi}) is equivalent to the existence of  a positive number $B$   and  a positive  number    $\eps _1$ such that  
 \begin{equation}\label{eq:jacobi1}
 \rho \Bigl(\vec{T_{x_t} L_t}, G_\varphi (x_t)\Bigr) \le  B \cdot t^2
 \end{equation}
 for  all $t \in (0, \eps_1)$ (Recall that $\vec{T_{x_t} L_t}$ is the unit $l$-vector  associated to the  oriented tangent space $T_xL$).
 \end{lemma}
 \begin{proof} Since  $\psi_0 = Id$  we observe  that   $|\pr \hat \varphi_t| = O (t^2)$  if and only
 $|\pr\hat \varphi _{| L_t}| =  O (t^2)$. Recall that we denoted by $\tilde \varphi (x)$ the form $\varphi$ (at the point $x$) regarded as a function on the Grassmannian $G_l(T_xM)$ of unit decomposable $l$-vectors. Then the function $|  d_w \tilde  \varphi (x)|$  is smooth in the variable  $w \in  G_l (T_xM)$. Since $d_w \varphi(x) = 0$  if $w \in G_\varphi(x)$, this implies that    there  exist  positive  constants   $C_1, C_2$ such that
 \begin{equation}\label{eq:taylor}
 C_1\cdot  |  d_w \tilde  \varphi (x)|  \le  \rho (w,  G_\varphi(x))  \le  C_2  \cdot  |  d_w \tilde  \varphi (x)|.
 \end{equation}
 Now, Lemma  \ref{lem:jacobi}  follows from (\ref{eq:prn}), which implies
\begin{equation}\label{eq:jacobi2}
   |\pr \hat \varphi _{ | T_{x_t} L_t}|  =    |  d_{\vec{T_{x_t} L_t}} \tilde  \varphi (x_t)|.
   \end{equation}
\end{proof}
Observe that $c_0 := \varphi^k(\vec{T_xL})$ is constant on $L$. Indeed, if $\varphi$ is parallel, then this follows from (\ref{eq:const}), and if $L$ is $\varphi$-calibrated, this holds for $c_0 :=1$ by definition. Thus, Lemma  \ref{lem:jacobi} implies that  there exist a constant  $C_3$  and $0 < \eps _2< \eps _1$  such  that for all $ x\in L$ and  all $t \in (0, \eps_2)$ we have
 \begin{equation}\label{eq:jacobi3}
c_0 - C_3  t^4 \le  \la  \varphi ,  \vec{T_{x_t} L_t}  \ra  \le  c_0 + C_3  t^4.
  \end{equation}
  Since $\psi_0 = Id$,   there exist  a constant  $C_4$  and a positive   number   $ \eps _3 < \eps_2$  such that, for all $ x\in L$ and  all $t \in (0, \eps_3)$, we have
\begin{equation}\label{eq:jacobi4}
1 - C_4 t  \le |(\psi_ t) _*  \vec{T_xL}|\le  1 + C_4  t.
\end{equation}
It follows from   (\ref{eq:jacobi3}) and (\ref{eq:jacobi4}) that   there  exists  a constant  $C_5$ such  that for all $x\in L$ and   all $t \in (0, \eps_3)$, we have
\begin{equation}\label{eq:jacobi6}
   (c_0- C_5 t ^3) \la  \varphi ,  (\psi_t)_* (\vec{T_x L})  \ra \le  |(\psi_ t) _*  \vec{T_xL}| \le (c_0+C_5 t^3) \la  \varphi , (\psi_t)_* (\vec{ T_x L})  \ra.
\end{equation}
Finally, it follows  from (\ref{eq:jacobi6}) and $\la \varphi ,  (\psi_t)_*(\vec{ T_x L})  \ra = \la \psi_t ^* (\varphi) \vec{T_xL} \ra$, that
\[
\frac{ d^2}{dt^2}{}_{| t =0}  vol  (\psi _t (L)) =  \int_L\frac{ d^2}{dt^2}{}_{| t =0} |(\psi_t)_* (\vec{T_xL})| \, dvol_L  = \int_L \frac{ d^2}{dt^2}{}_{| t =0}  (\psi_t)^* \varphi = 0.
\]
Hence  $s$ is a Jacobi    vector field.  This proves  (\ref{eq:compareJ}).
  
For the last statement, assume that $s$ is a Jacobi vector field on the $\varphi$-calibrated submanifold $L$. 
By Remark   2.3 in \cite{LV2017}, 
 $s$ is an infinitesimal deformation of $L$ as a $\varphi$-calibrated  submanifold.  This
 is the same   to  say that (\ref{eq:jacobi1}) holds for some $B$ and $\eps_0$.
 By Lemma \ref{lem:jacobi},  this  implies  that $s \in J_{\hat \varphi} (L)$.
 \end{proof}
 
Note that (\ref{eq:compareJ}) in case of hypothesis (2) in Proposition \ref{prop:infinitesimal} can also be seen by an argument along the lines of the proof of Theorem \ref{thm:crit}.
 
Recall that $J(L)$ is finite dimensional and its dimension is called  the {\it nullity  of $L$} \cite{Simons1968}.

From Proposition \ref{prop:infinitesimal} we obtain Corollary \ref{cor:jacobi} below,  which has been   first proved by Simons  in \cite[Theorem 3.5.1]{Simons1968}  by computing the Jacobi operator 
on a compact  K\"ahler submanifold $L$. Simons'  computation has been  generalized  by McLean \cite{McLean1998} for  calibrated submanifolds,  and    simplified by L\^e-Van\v zura \cite{LV2017}, using different  methods.

\begin{corollary}\label{cor:jacobi}  Let $L$ be a compact and closed  K\"ahler submanifold
in a  K\"ahler  manifold $(M, g, \om^2)$. Then
 the nullity of $L$ is equal to the dimension of the space
of globally defined holomorphic sections in  $NL$.
\end{corollary} 

We are now ready to show the main result of this section.

\begin{theorem}[Main Theorem]\label{thm:main}  Let $\varphi \in \Om^l(M)$ be a parallel calibration on a real analytic Riemannian manifold $(M, g)$, let $\alpha \in H_l(M, {\mathbb Z})$ be a homology class such that $\langle [\varphi], \alpha\rangle \neq 0$, and let $L \in \Mm_{\hat \varphi}(\alpha)$. 
\begin{enumerate}
\item The  pre-moduli space $\Mm_{\hat \varphi}(\alpha)$   is a finite dimensional analytic space, and so is, in particular, $\Mm_{\hat \varphi}(L) \subset \Mm_{\hat \varphi}(\alpha)$.
\item If $L$ is $\varphi$-calibrated, then a formally   unobstructed   Jacobi    field $s  \in  J(L)$
is smoothly unobstructed. 
\item If $L$ is $\varphi$-calibrated, then there is a canonical $\Z_2$-graded strongly homotopy Lie algebra that governs  formal and smooth deformations of $L$ in the  class of  $\varphi$-calibrated    submanifolds. 
\end{enumerate}
\end{theorem}

\begin{proof}
Let $\Nn_\varphi := \{ v \in TM \mid \imath(v) \varphi = 0\}$ be the annihilator of $\varphi$ which induces the parallel orthogonal decomposition $TM = \Nn_\varphi \oplus \Nn_\varphi^\perp$, and clearly, the restriction of $\hat \varphi$ to $\Nn_\varphi^\perp$ is multi-symplectic. As $L$ is $\varphi$-calibrated, it follows that $TL \subset \Nn_\varphi^\perp$, so that $L$ is contained in a maximal leaf $M_0 \subset M$ of the (parallel and hence integrable) distribution $\Nn_\varphi^\perp$. The normal bundle of $L$ decomposes orthogonally as
\begin{equation} \label{eq:normal-L}
NL = NL_0 \oplus (\Nn_\varphi)_{|L},
\end{equation}
where $NL_0$ is the normal bundle of the inclusion $L \hookrightarrow M_0$. As $L$ is closed, there is an $\eps > 0$ such that the normal exponential $\exp: NL_\eps \to M$ is a fiberwise local diffeomorphism, where $NL_\eps \subset NL$ denotes the $\eps$-disc bundle. Let $g_N := \exp^\ast(g)$ be the induced metric on $NL_\eps$. As $g$ is a local product metric, the $g_N$-orthogonal complement of the fibers of $((\Nn_\varphi)_{|L})_\eps \to L$ induce a flat connection on this disc bundle, and decomposing a small section $s = s_0 + s_1 \in \Om^0(L, NL_\eps) = \Om^0(L, (NL_0)_\eps) \oplus \Om^0(L,((\Nn_\varphi)_{|L})_\eps)$ according to (\ref{eq:normal-L}), the definition of $F_{\hat \varphi}$ implies that $F_{\hat \varphi}(s) = 0$ iff $F_{\hat \varphi}(s_0) = 0$ and $s_1$ is parallel. Thus,
\[
\Mm_{\hat \varphi}(L) \cong \Mm^0_{\hat \varphi}(L) \times \R^k,
\]
where $\Mm^0_{\hat \varphi}(L)$ is the premoduli space of $L \subset M_0$, and $k$ is the dimension of the space of parallel sections in $\Om^0(L,(\Nn_\varphi)_{|L})$.

With this, it suffices to show the theorem for $L \subset (M_0, \varphi_{|M_0})$, and, after replacing $M_0$ by $M$, we may therefore assume w.l.o.g. that $\varphi$ is multi-symplectic.

Being parallel, $\varphi$ is  harmonic and hence  analytic. If $L \in \Mm_{\hat \varphi}(\alpha)$, then $\langle [\varphi], \alpha\rangle \neq 0$ implies that $\varphi_{|L} \neq 0$ and moreover, by Theorem \ref{thm:crit}, (\ref{eq:const}) is satisfied, and $L \subset M$ is a minimal submanifold. In particular, $L$ is analytic by the Morrey regularity theorem \cite{Morrey1954, Morrey1958, Morrey2008}.

Therefore, Proposition \ref{prop:infinitesimal} implies that $J_{\hat \varphi}(L) \subset J(L)$ is  finite dimensional, so that $\Mm_{\hat \varphi}(L)$ is an analytic space by Theorem \ref{thm:Main-Psi}. Since this is the case for any $L \in \Mm_{\hat \varphi}(\alpha)$, it follows that $\Mm_{\hat \varphi}(\alpha)$ is an analytic space as well.
 
The second assertion  of Theorem  \ref{thm:main} is a corollary  of the first  assertion, Proposition \ref{prop:infinitesimal}, and the Artin's approximation theorem \cite[Theorem 1.2]{Artin1968}, which implies that,  in a finite dimensional analytic space,   smooth and   formal    obstructedness  are equivalent.
  
For the last statement, assume that $L$ is   a $\varphi^{2l}$-calibrated submanifold.
Then  the last   assertion of Theorem \ref{thm:main}  for  $L$  follows from the second assertion and  Proposition \ref{prop:mc}.

 Now  assume that   $L$ is a $\varphi^{2l-1}$-calibrated submanifold. Then  $L \times S^1$ is a $\varphi^{2l-1}\wedge  dt$-calibrated  submanifold  in $(M\times S^1, g + dt^2, \varphi \wedge dt)$.  It is not hard to see  that, if   $\tilde L_t$ is a smooth deformation  of $L\times S^1$ in the class of  minimal submanifolds in $M \times S^1$, then
 $\tilde L_t = L_t \times  S^1$ for some family of $\varphi$-calibrated  submanifolds $L_t$.  Hence,  the formal and smooth  deformations
 of  $\varphi^{2l-1}$-calibrated submanifold are governed   by the   $\Z_2$-graded   strongly homotopy Lie algebra associated to $L\times S^1$.
 This completes  the proof of Theorem \ref{thm:main}.
\end{proof}

\begin{remark}\label{rem:simul}   

1.  Theorem  \ref{thm:main}   is also valid  for  open   $\varphi$-calibrated submanifolds  with    compactly supported variation fields.

2.   Assume that $L$  is     simultaneously  a $\varphi$-calibrated submanifold  and a $\varphi'$-calibrated submanifold,  where
$\varphi$ and $\varphi'$ are    calibrations  on $(M, g)$. Then   any $\varphi$-calibrated  closed  submanifold  $L'$ that is  homologous to $L$ is also 
a $\varphi'$-calibrated submanifold. This    implies   that deformations of such  calibrated submanifolds are easier to    control.
For example,  let  $(M^6, g, \om^2, \alpha = \Re vol_\C)$   be a Calabi-Yau 6-manifold  and $C \subset  (M^6, g, \om^2, \alpha)$   a complex curve. Clearly,  the product  $L : = S^1\times C$       is   simultaneously   calibrated  with respect to both the associative calibration  $\varphi: = dt \wedge \om^2 +\alpha$ and the  calibration $dt \wedge \om^2$.
Hence, any  deformation  $L'$ of  $L$ in the class of   associative     submanifolds is also calibrated by $dt\wedge \om$. In particular, $L'$ is invariant under  the  flow generated by  the vector field $\p _{t}$.  This flow  preserves  the Calabi-Yau structure  on each  slice  $\{ t \} \times M^6$.  We conclude that all the slices $L' \cap \{ t = \text{constant}\}$  are isomorphic as  complex curves in $M^6$.
It follows  that $L' =  S^1 \times  C'$, where $C'$ is a     complex deformation of $C$. In particular, if $C$ is isolated then $L$ is isolated. The last assertion  has been obtained
in \cite[Lemma 5.11]{CHNP2012} by computing  the  kernel of the corresponding  linearized operators  that control  the  corresponding deformations.
In \cite{Leung2002}  Leung  studies   deformations of simultaneously calibrated submanifolds using  integral estimates.  We refer the interested reader to \cite{Le1993} for  the  relation between the calibration method  and the integral  estimate method in the theory of minimal submanifolds.

3.  As  we noted in Corollary  \ref{cor:linfty},  there are two natural   $\Z_2$-graded  strongly homotopy Lie algebras associated  to    an associative
submanifold $L$ in a $G_2$-manifold $(M^7, g, \varphi)$.  It is known that the smooth and infinitesimal $\chi$-deformations of  $L$   coincide with the smooth  and infinitesimal deformations of $L$ as  a minimal submanifold \cite{McLean1998, LV2017}. Thus, the  strong homotopy Lie algebra attached to $L$ via  $\chi$ also governs  smooth  and formal  deformations  of  $L$ as a $\varphi$-calibrated  submanifold.

4. The action of $\Di_\Psi(M)$ preserves the analytic structure on $\Mm_\Psi(L)$,
whence the moduli space $\Mm_\Psi(L)/\Di_\Psi(M )$ is an analytic space as
well. In particular, generic points of $\Mm_\Psi(L)$ or $\Mm_\Psi(L)/\Di_\Psi(M )$, respectively, are smooth, and hence (formally) unobstructed.
\end{remark}

\subsection{Deformations of complex submanifolds revisited}\label{subs:compl}

\begin{theorem}\label{thm:c}  Assume that $L$ is a closed complex submanifold in a  complex manifold $(M,J)$. 
\begin{enumerate}
\item The pre-moduli space $\Mm_J(L)$ with the $C^1$-topology has the  structure  of  a finite   dimensional analytic space.
\item A formally   unobstructed   holomorphic normal  field 
is smoothly unobstructed.  
\item There is a canonical $\Z$-graded strongly homotopy Lie algebra that governs  formal and smooth complex deformations of $L$. 
\end{enumerate}
\end{theorem}
\begin{proof}[Proof  of Theorem \ref{thm:c}] Theorem   \ref{thm:c} is proved in the same way as   Theorem \ref{thm:main}  and  we omit its proof.
\end{proof}

\begin{remark}\label{rem:compl}  Deformations of complex     submanifolds   have been    examined by Ji  in \cite{Ji2014}, using  his general theory of deformations  of   Lie subalgebroids.    Since   the  Fr\"olicher-Nijenhuis bracket   of $J \in \Om^*(M, TM)$   is $-i/2$-times the  Dolbeault  operator $\bar \p$,
 Ji's  strongly homotopy   Lie algebra  is the same as ours up to an uninfluential global factor (see  also \cite{Manetti2007}  for an equivalent  formulation).
\end{remark}
 
 \
 
 \subsection{Conclusion} In our  paper, inspired   by the principle
  that every deformation problem over a field of characteristic zero  is governed by  a differential graded Lie algebra (or, equivalently, by an $L_\infty$-algebra),  we  found   new connections between
 seemingly unrelated    subjects:     coisotropic  submanifolds  in Jacobi
  manifolds, calibrated and  more  generally  $\varphi$-critical  submanifolds, complex   submanifolds  in complex manifolds. As a consequence of this deformation theory approach, we  and other  authors were then lead to discover  new  structures  defined  by the Fr\"olicher-Nijenhuis  bracket on  $G_2$  and   Spin(7)-manifolds  \cite{KLS2017a, KLS2017b, KLS2018}, \cite{CKT2018}.   
 \par 
  In the present paper, we also  give  a new  treatment of Jacobi  vector  fields  on calibrated  submanifolds (Proposition \ref{prop:infinitesimal}), generalizing results by Simons, McLean  and L\^e-Van\v zura,   prove  a   result  on the  structure  of   pre-moduli  space  of  $\varphi$-calibrated submanifolds, where $\varphi$  is a parallel  calibration  in a  real analytic  Riemannian manifold (Theorem \ref{thm:main}), and 
  a similar  theorem    for  $\Psi$-submanifolds under  a  certain condition (Theorem \ref{thm:Main-Psi}), which generalizes
 a known  result  on deformations of complex submanifolds in complex
  manifolds.

\subsection*{Acknowledgement} 
LV and LS warmly thank HVL and the Institute of Mathematics of the Czech Academy of Sciences for their hospitality during their respective stays in Prague, where part of this project was developed. LV is member of the GNSAGA of INdAM. We also thank the referees for their careful reading of this manuscript, which helped us to correct some inaccuracies and to significantly improve its exposition.

\end{document}